\documentclass[12pt]{amsart}

\usepackage{amsmath,amssymb,ifthen}
\usepackage{fullpage}
\usepackage{graphicx,psfrag,subfigure}
\usepackage{color}
\usepackage{moreverb}

\def\R{{\mathbb R}}
\def\N{{\mathbb N}}

\def\KK{{\mathcal K}}
\def\NN{{\mathcal N}}

\def\OO{{\mathcal O}}
\def\PP{{\mathcal P}}
\def\SS{{\mathcal S}}
\def\TT{{\mathcal T}}
\def\XX{{\mathcal X}}

\def\diam{{\rm diam}}
\def\edual#1#2{\langle\hspace*{-1mm}\langle#1\,,\,#2\rangle\hspace*{-1mm}\rangle}
\def\norm#1#2{\|#1\|_{#2}}

\def\enorm#1{|\hspace*{-.5mm}|\hspace*{-.5mm}|#1|\hspace*{-.5mm}|\hspace*{-.5mm}|}
\def\set#1#2{\big\{#1\,:\,#2\big\}}
\def\eps{\varepsilon}
\def\dual#1#2{\langle#1\,,\,#2\rangle}

\def\gal{{\mathcal{G}}}
\def\level{{\rm level}}
\def\QQ{{\mathcal{Q}}}

\def\normal{\boldsymbol{n}}

\def\dlp{\mathcal{K}} % double-layer potential
\def\hyp{\mathcal{W}} % hypersingular integral operator
\def\prec{\mathcal{P}} % Preconditiner operator
\def\PAS{\prec_{\rm AS}}
\def\PAStilde{\widetilde\prec_{\rm AS}}
\def\PPAS{\mathbf{P}_{\rm AS}}
\def\PPAStilde{\widetilde{\mathbf{P}}_{\rm AS}}
\def\AA{\mathbf{A}}
\def\PP{\mathbf{P}}
\def\BB{\mathbf{B}}
\def\CC{\mathbf{C}}
\def\DD{\mathbf{D}}
\def\evmin{\lambda_{\rm min}}
\def\evmax{\lambda_{\rm max}}
\def\xx{{\mathbf{x}}}
\def\yy{{\mathbf{y}}}

\def\bb{{\bf b}}
\def\cond{{\rm cond}}
\def\hmax#1{h_{{\rm max},#1}}
\def\hmin#1{h_{{\rm min},#1}}

\def\linhull{{\rm span}} 

\def\supp{{\rm supp}}
\newcounter{constantsnumber}
\def\namec#1#2{%
  \ifthenelse{\equal{#1}{lipschitz}}{C_{\rm lip}}{%
  \ifthenelse{\equal{#1}{monotone}}{C_{\rm mon}}{%
  \ifthenelse{\equal{#1}{cea}}{C_{\mbox{\rm\scriptsize C\'ea}}}{%
  \ifthenelse{\equal{#1}{norm}}{C_{\rm norm}}{%
  \ifthenelse{\equal{#1}{mon}}{{C}_{\rm mon}}{
  \ifthenelse{\equal{#1}{lip}}{{C}_{\rm lip}}{
  \ifthenelse{\equal{#1}{monA}}{c_{\rm mon}}{
  \ifthenelse{\equal{#1}{lipA}}{c_{\rm lip}}{
  \ifthenelse{\equal{#1}{normequiv1}}{c_{\rm norm}}{ 
  \ifthenelse{\equal{#1}{inv}}{C_{\rm inv}}{ 
  \ifthenelse{\equal{#1}{inv2}}{\widetilde{C}_{\rm inv}}{ 
  \ifthenelse{\equal{#2}{newcounter}}{\refstepcounter{constantsnumber}\label{const#1}}{}C_{\ref{const#1}}}%
  }}}}}}}}}}}
\def\setc#1{\namec{#1}{newcounter}}
\def\c#1{\namec{#1}{reference}}

\newtheorem{theorem}{Theorem}
\newtheorem{proposition}[theorem]{Proposition}
\newtheorem{lemma}[theorem]{Lemma}

\def\subsection#1
{
 \bigskip

 \refstepcounter{subsection}
 {\noindent\bf\arabic{section}.\arabic{subsection}.~#1.~}
}

\begin{document}%%%%%%%%%%%%%%%%%%%%%%%%%%%%%%%%%%%%%%%%%%%%%%%%%%%%%

%%%%%%%%%%%%%%%%%%%%%%%%%%%%%%%%%%%%%%%%%%%%%%%%%%%%%%%%%%%%%%%%%%%%%
% Title and Authors
\title[Efficient additive Schwarz preconditioning for hypersingular integral
equations on locally refined triangulations]{Efficient additive Schwarz preconditioning for hypersingular integral equations 
\\on locally refined triangulations}
\date{\today}

\author{Michael Feischl}
\author{Thomas F\"uhrer}
\author{Dirk Praetorius}
\author{Ernst P.~Stephan}

\thanks{{\bf Acknowledgement.} The research of the authors is supported 
by the Austrian Science Fund (FWF) through the research project 
{\em Adaptive boundary element method}, funded under grant P21732.}

%%%%%%%%%%%%%%%%%%%%%%%%%%%%%%%%%%%%%%%%%%%%%%%%%%%%%%%%%%%%%%%%%%%%%
% Classification
\keywords{preconditioner, multilevel additive Schwarz,
hypersingular integral equation}
\subjclass[2010]{65N30, 65F08, 65N38}
%%%%%%%%%%%%%%%%%%%%%%%%%%%%%%%%%%%%%%%%%%%%%%%%%%%%%%%%%%%%%%%%%%%%%
% ABSTRACT
\begin{abstract}
For non-preconditioned Galerkin systems, the condition number grows
with the number of elements as well as the quotient of the maximal and the
minimal mesh-size. Therefore, reliable and effective numerical 
computations, in particular on adaptively refined meshes, require the
development of appropriate preconditioners.
We analyze and numerically compare multilevel additive Schwarz preconditioners
for hypersingular integral equations, where 2D and 3D as well as closed
boundaries and open screens are covered. The focus is on a new local
multilevel preconditioner which is optimal in the sense that 
the condition number of the corresponding preconditioned system is 
independent of the number of elements, the local mesh-size, and the 
number of refinement levels.
\end{abstract}
%%%%%%%%%%%%%%%%%%%%%%%%%%%%%%%%%%%%%%%%%%%%%%%%%%%%%%%%%%%%%%%%%%%%%
% Make Title
\maketitle
%%%%%%%%%%%%%%%%%%%%%%%%%%%%%%%%%%%%%%%%%%%%%%%%%%%%%%%%%%%%%%%%%%%%%
% CONTENTS
% !TEX root = hypsingAS.tex

\section{Introduction}

\noindent
Let $\Omega\subset\R^d$ be a bounded polygonal resp.\ polyhedral Lipschitz 
domain in $\R^d$, $d=2,3$, with connected boundary $\Gamma = \partial\Omega$. 
For a given right-hand side $f$, we consider the hypersingular integral equation
\begin{align}\label{eq:hypsing}
  \hyp u (x) := -\partial_{\normal_x} \int_\Gamma \partial_{\normal_y} G(x-y)
  u(y) \,ds_y
  = f(x) \quad\text{for }x\in\Gamma.
\end{align}
Here, $\partial_{\normal_x}$ is the normal derivative with respect to
$x\in\Gamma$, and $G(z)$ denotes the fundamental solution of the Laplacian
\begin{align}
  G(z) = \begin{cases}
    -\frac1{2\pi} \log|z| &\text{for } d=2, \\
    +\frac1{4\pi} \frac1{|z|} &\text{for } d=3.
  \end{cases}
\end{align}
The exact solution $u$ of~\eqref{eq:hypsing} cannot be computed analytically in general. For a given
triangulation $\TT_\ell$ of $\Gamma$, one can e.g.\ use the Galerkin boundary element
method (BEM) to compute an approximation $u_\ell$ of $u$ instead. 
If a certain accuracy of the approximation $u_\ell\approx u$ is required, 
adaptive mesh-refining algorithms of the type
\begin{align*}
 \boxed{\texttt{ SOLVE }}
 \quad\longrightarrow\quad
 \boxed{\texttt{ ESTIMATE }}
 \quad\longrightarrow\quad
 \boxed{\texttt{ MARK }}
 \quad\longrightarrow\quad
 \boxed{\texttt{ REFINE }}
\end{align*}
are used, where, starting with a given initial triangulation $\TT_0$, a sequence of locally refined triangulations $\TT_\ell$ and corresponding Galerkin
solutions $u_\ell$ are computed. The lowest-order BEM for~\eqref{eq:hypsing} uses $\TT_\ell$-piecewise
affine and globally continuous functions $u_\ell\in\SS^1(\TT_\ell)$ to approximate
$u$, and the adaptive mesh-refinement leads to a nested sequence of spaces
$\SS^1(\TT_\ell)\subset\SS^1(\TT_{\ell+1})$ for all $\ell\ge0$. 

In recent years, the convergence of adaptive BEM even with quasi-optimal algebraic
rates has been proved~\cite{fkmp,gantumur,part1,part2}. 
Throughout, it is however assumed that the Galerkin solution
$u_\ell$ is computed exactly, i.e.\ the resulting linear system 
$\AA^\ell \xx^\ell = \bb^\ell$ is solved exactly. As is well known, the accuracy
of direct solvers as well as the effectivity of iterative solvers is usually 
spoiled by the conditioning of the matrix $\AA^\ell$. For uniform triangulations 
$\TT_\ell$ with number of elements $N_\ell = \# \TT_\ell$, it holds
$\cond_2(\AA^\ell) \lesssim N_\ell^{1/(d-1)}$ for the $\ell_2$-condition number.
For adaptively refined triangulations $\TT_\ell$ with maximal element diameter
$\hmax\ell$ and minimal element diameter $\hmin\ell$ the situation is even
worse~\cite{amt99}, namely $\cond_2(\AA^\ell) \lesssim N_\ell (1 + |\log(N_\ell
\hmin\ell)|)$ for $d=2$ resp. $\cond_2(\AA^\ell) \lesssim N_\ell^{1/2}
(\hmax\ell / \hmin\ell)^2$ for $d=3$.

Therefore, reliable and effective numerical computations require the development
of efficient preconditioners. Prior work includes diagonal scaling of the BEM
matrices which reduces the condition number for adaptive triangulations down to
that of a uniform triangulation with the same number of elements~\cite{amt99,gm06}. 
%at least for $d=3$.
Other preconditioners for the Galerkin BEM of hypersingular integral equations
are proposed in~\cite{tsm97,sw98,tsz98,cao} and the references 
therein, where mainly quasi-uniform triangulations are thoroughly analyzed.

Our work focuses on additive Schwarz preconditioners for the Galerkin BEM 
of~\eqref{eq:hypsing} with lowest-order polynomials $u_\ell\in\SS^1(\TT_\ell)$. 
For uniform triangulations, it is shown in~\cite{transtep96} that this approach leads to 
bounded condition numbers for the preconditioned system, i.e.\
$\cond((\BB^\ell)^{-1}\AA^\ell)\le C<\infty$ with some $\ell$-independent constant
$C>0$.
The same is proved for partially adapted triangulations in~\cite{amcl03}, where it
is assumed that $\TT_{\ell}\cap\TT_{\ell+1}\subset\TT_{\ell+k}$ for all $\ell,k\in\N_0$, 
i.e.\ as soon as an element $T\in\TT_\ell$ is not refined, it remains non-refined 
in all succeeding triangulations. In our contribution, we 
remove such an assumption which is infeasible in practice, and only rely on 
nestedness $\SS^1(\TT_\ell)\subset\SS^1(\TT_{\ell+1})$ of the discrete ansatz
spaces. 
The main idea is to use only new nodes in $\TT_{\ell+1}
\backslash \TT_\ell$ plus their neighbouring nodes for preconditioning. 
In the frame of 2D FEM problems, such an idea has already been considered 
in the works~\cite{mitchell,wuchen06,xch10}. For a V-cycle multigrid method,
stability  for the subspace decomposition in $H^1$ has been proved 
in~\cite{wuchen06} by means of a variant of the Scott-Zhang
projection~\cite{sz}.
In our work, we extend these results to the fractional-order Sobolev space $H^{1/2}$.

First, we give the analysis for the case $\Gamma=\partial\Omega$ and a 
stabilized Galerkin formulation which factors the constant functions out. We 
stress that the results of this work also apply to screens 
$\Gamma\subsetneqq\partial\Omega$, and the corresponding analysis is obtained
by simply omitting all stabilization related terms. We also refer to the short
Section~\ref{sec:screen} for further remarks.

While all constants and their
dependencies are explicitly given in all statements, in proofs we use the symbol 
$\lesssim$ to abbreviate $\le$ up to some multiplicative constant which is clear 
from the context. Moreover, we use $\simeq$ to abbreviate that both estimates
$\lesssim$ and $\gtrsim$ holds.

The remainder of this work is organized as follows:
Section~\ref{sec:main} contains the analytical main result of this 
work. We first recall the necessary notation to define the new local 
multilevel preconditioner and then formulate the main result
(Theorem~\ref{thm:main}). 
Furthermore, we introduce a global multilevel preconditioner and give a
similar but weaker result (Theorem~\ref{thm:gmld}).
For the ease of presentation, we first focus
on closed boundaries.
Section~\ref{sec:examples} formulates other preconditioners.
Numerical experiments on closed boundaries and slits in 2D compare the
condition numbers of the corresponding preconditioners as well as those
with no preconditioning. In both, theory
and practice, the new local multilevel preconditioner proves to be optimal
in the sense that the condition number remains uniformly bounded which
is not the case for the other strategies considered.
In Section~\ref{section:proof} we give a proof of Theorem~\ref{thm:main} and in
Section~\ref{section:proof:gmld} we give a proof of Theorem~\ref{thm:gmld}.
The final Section~\ref{sec:screen} comments on extension of the
analysis to open screens in 2D and 3D.
We show that the main result for the local and global multilevel preconditioner also
hold for problems on open screens $\Gamma \subsetneqq \partial\Omega$
in 2D and 3D.

% !TEX root = hypsingAS.tex

\section{Main result}
\label{sec:main}
%--------------------------------------------------------------------------------
\subsection{Continuous setting}
%--------------------------------------------------------------------------------
Let $\Gamma:=\partial\Omega$.
By $H^s(\Gamma)$, we denote the usual Sobolev spaces which are, e.g., given
by real interpolation $H^s(\Gamma) = [L^2(\Gamma);H^1(\Gamma)]_s$, for 
$0<s<1$, and we let $H^0(\Gamma):=L^2(\Gamma)$. The space
$H^{-s}(\Gamma) := H^s(\Gamma)^*$ is the dual space of $H^s(\Gamma)$, where
duality is understood with respect to the extended $L^2(\Gamma)$-scalar 
product $\dual\cdot\cdot_\Gamma$.

It is known that $\hyp$ induces a linear and bounded operator 
$\hyp:H^{s}(\Gamma)\to H^{s-1}(\Gamma)$, for all $0\le s\le 1$ which is
symmetric and positive semidefinite on $H^{1/2}(\Gamma)$. Moreover,
it holds $\dual{\hyp v}{1}_\Gamma=0$. Let 
$H_0^{\pm1/2}(\Gamma):=\set{v\in H^{\pm1/2}(\Gamma)}{\dual{v}{1}_\Gamma=0}$.
We thus suppose that the right-hand side $f$ in~\eqref{eq:hypsing} satisfies $f\in H^{-1/2}_0(\Gamma)$.

As $\Gamma$ is connected, the kernel of $\hyp$ are precisely the constant 
functions, and thus $\hyp:H_0^{1/2}(\Gamma)\to H_0^{-1/2}(\Gamma)$ is linear, 
continuous, symmetric, and elliptic. By virtue of the Rellich compactness 
theorem, the definition
\begin{align}
 \edual{v}{w} := \dual{\hyp v}{w}_\Gamma + \dual{v}{1}_\Gamma\dual{w}{1}_\Gamma
\end{align}
provides a scalar product on $H^{1/2}(\Gamma)$, and the induced norm
$\enorm{v}^2 := \edual{v}{v}$ is equivalent to the usual $H^{1/2}(\Gamma)$-norm.
In particular, the hypersingular integral equation~\eqref{eq:hypsing} is equivalently recast
in the variational formulation
\begin{align}\label{eq:weakform}
 \edual{u}{v} = \dual{f}{v}_\Gamma
 \quad\text{for all }v\in H^{1/2}(\Gamma).
\end{align}
According to the Lax-Milgram lemma, this formulation allows for a unique
solution $u\in H^{1/2}(\Gamma)$. Due to $f\in H^{-1/2}_0(\Gamma)$, it follows
$u\in H^{1/2}_0(\Gamma)$.

%--------------------------------------------------------------------------------
\subsection{Triangulation and general notation}\label{section:main:triangulation}
%--------------------------------------------------------------------------------
Let $\TT_\ell$ denote a regular triangulation of $\Gamma$ into compact
affine line segments ($d=2$) resp.\ compact plane surface triangles ($d=3$).
We define the local mesh-width function $h_\ell \in L^\infty(\Gamma)$ by
\begin{align}
  h_\ell|_T := h_\ell(T) := \diam(T) \quad\text{for all }T\in\TT_\ell.
\end{align}
We suppose that $\TT_\ell$ is $\gamma$-shape regular in the sense that
\begin{align}\label{eq:shaperegular}
  h_\ell(T) \leq \gamma \, h_\ell(T') 
  \quad\text{and}\quad
  h_\ell(T) \leq \gamma \, |T|^{1/(d-1)}
\end{align}
for all $T,T'\in\TT_\ell$ with $T\cap T'\neq\emptyset$.
Here and throughout, $|T|$ denotes the $(d-1)$-dimensional surface measure of 
$T\in\TT_\ell$ and hence $|T| = \diam(T)$ for $d=2$. We note that for either
dimension $d=2,3$, one of the conditions in~\eqref{eq:shaperegular} is automatically satisfied.

We consider lowest-order conforming boundary elements, where
\begin{align}\label{eq:defdiscretespace}
 \XX^\ell := \SS^1(\TT_\ell) := \set{v\in C(\Gamma)}{v|_T \text{ is affine for
 all }T\in\TT_\ell}.
\end{align}
Let $\NN_\ell$ denote the set of nodes of the mesh 
$\TT_\ell$. The natural basis of $\XX^\ell$ is given by the hat-functions. 
For each node $z\in \NN_\ell$, let $\eta_z^\ell\in\SS^1(\TT_\ell)$ be the 
hat-function characterized by
\begin{align}
  \eta_z^\ell(z) = 1 \quad\text{and}\quad \eta_z^\ell(z') = 0
  \quad\text{for all } z'\in\NN_\ell \backslash \{z\}.
\end{align}

For any subset $\tau\subseteq\Gamma$, we define the patch
$\omega_\ell^k(\tau)\subseteq \Gamma$ inductively by
\begin{align}\label{eq:patch}
 \omega_\ell^1(\tau) := \omega_\ell(\tau)
 := \bigcup\set{T\in\TT_\ell}{T\cap\tau\neq\emptyset},
 \quad
 \omega_\ell^{k+1}(\tau) := \omega_\ell^k(\omega_\ell(\tau))
 \quad\text{for }k\in\N.
\end{align}
For any node $z\in\TT_\ell$, we abbreviate $\omega_\ell^k(z):=\omega_\ell^k(\{z\})$ and note that $\omega_\ell(z) = \supp(\eta_z^\ell)$.

As in~\cite{wuchen06}, we further
define for every node $z\in\NN_\ell$ the mesh-width $h_\ell(z)$ as the 
shortest edge $E$ of $\TT_\ell$ with $z\in E$.
It holds
\begin{align}
  h_\ell(z) \le h_\ell(T) \lesssim h_\ell(z)
  \quad\text{for all }z\in\NN_\ell\text{ and }
  T\in\TT_\ell \text{ with } z\in T,
\end{align}
where the hidden constant depends only on the $\gamma$-shape regularity of
$\TT_\ell$. The hat-functions satisfy
\begin{align}\label{eq:hatfunprop}
\begin{split}
  0\leq \eta_z^\ell \leq 1,
  \quad
  \norm{\nabla \eta_z^\ell}{L^\infty(\Gamma)} \lesssim h_\ell^{-1}(z),
  \quad\text{and}\quad
  \sum_{z\in\NN_\ell} \eta_z^\ell = 1,
\end{split}
\end{align}
where the hidden constant depends only on the $\gamma$-shape regularity of
$\TT_\ell$.

%--------------------------------------------------------------------------------
\subsection{Galerkin discretization}
%--------------------------------------------------------------------------------
The Galerkin solution $u_\ell\in\XX^\ell$ to the solution
$u$ of~\eqref{eq:weakform} solves
\begin{align}\label{eq:galerkin}
 \edual{u_\ell}{v_\ell} = \dual{f}{v_\ell}_\Gamma
 \quad\text{for all }v_\ell \in \XX^\ell.
\end{align}
Fixing a numbering of the nodes $\NN_\ell = \{z_1,\dots,z_N\}$, the
discrete solution $u_\ell$ from~\eqref{eq:galerkin} is obtained by solving
a linear system of equations $\AA^\ell \xx^\ell = \bb^\ell$ in $\R^N$, where
\begin{align}\label{eq:galerkinmatrix}
 \AA^\ell_{jk} = \edual{\eta_{z_k}^\ell}{\eta_{z_j}^\ell},
 \quad
 \bb^\ell_j = \dual{f}{\eta_{z_j}^\ell}_\Gamma,
 \quad\text{and}\quad 
 u_\ell = \sum_{k=1}^N \xx^\ell_k\eta_{z_k}^\ell.
\end{align}

\begin{figure}[t]
 \centering
 \psfrag{T0}{}
 \psfrag{T1}{}
 \psfrag{T2}{}
 \psfrag{T3}{}
 \psfrag{T4}{}
 \psfrag{T12}{}
 \psfrag{T34}{}
 \includegraphics[width=35mm]{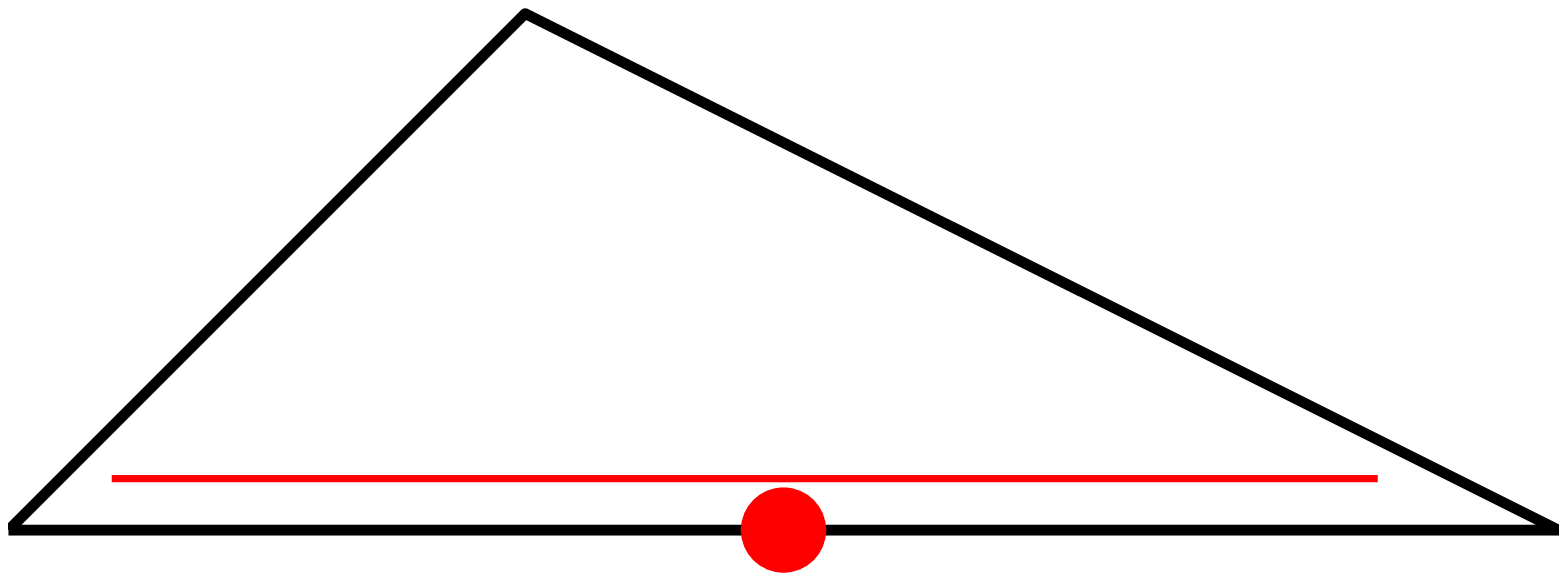} \quad
 \includegraphics[width=35mm]{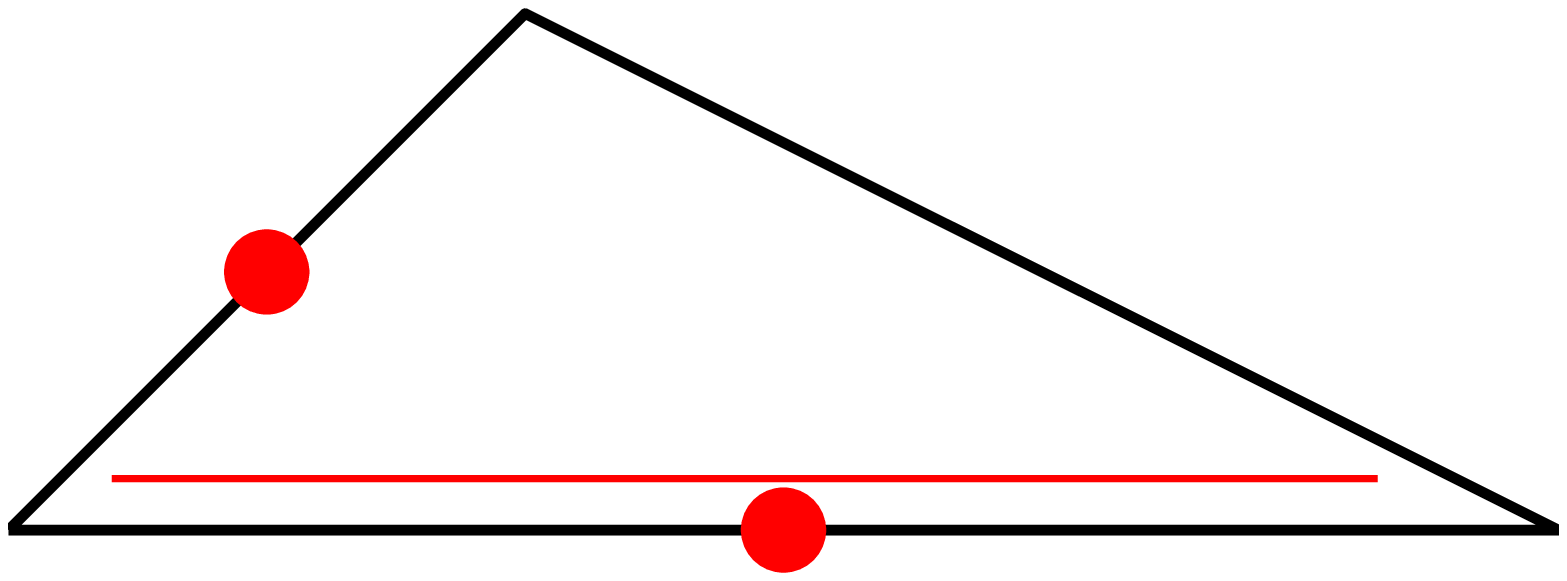} \quad
 \includegraphics[width=35mm]{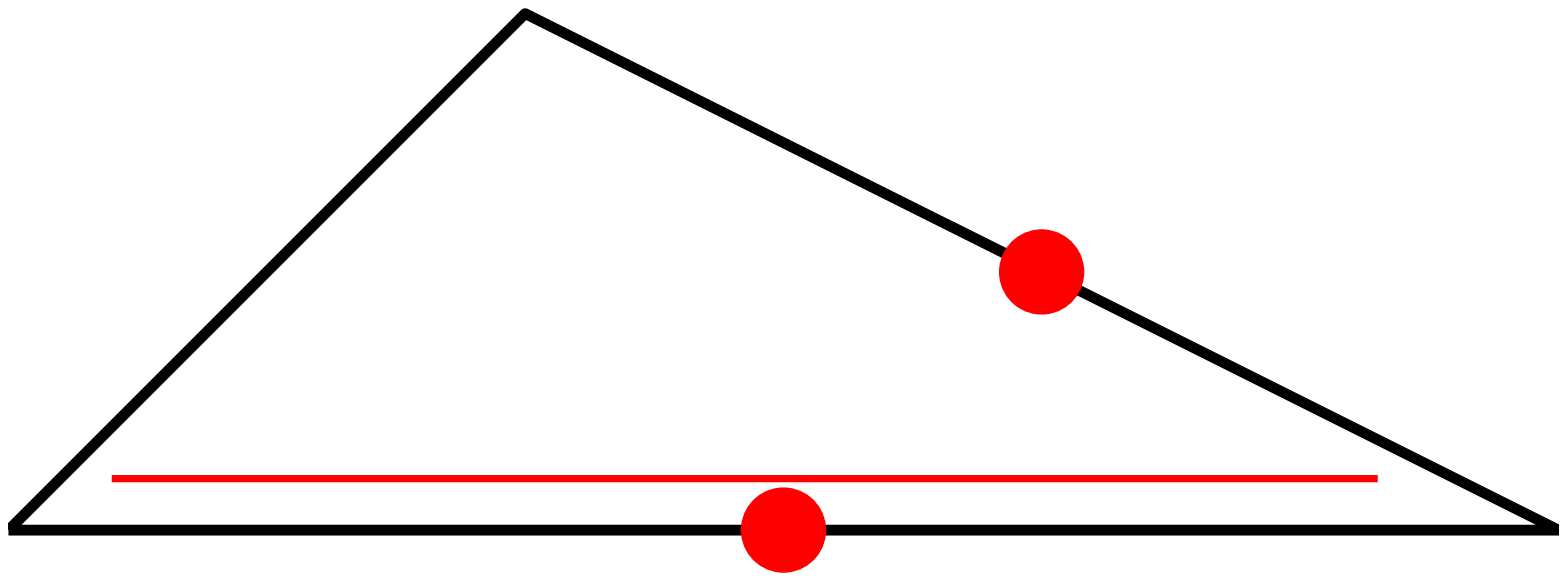} \quad
 \includegraphics[width=35mm]{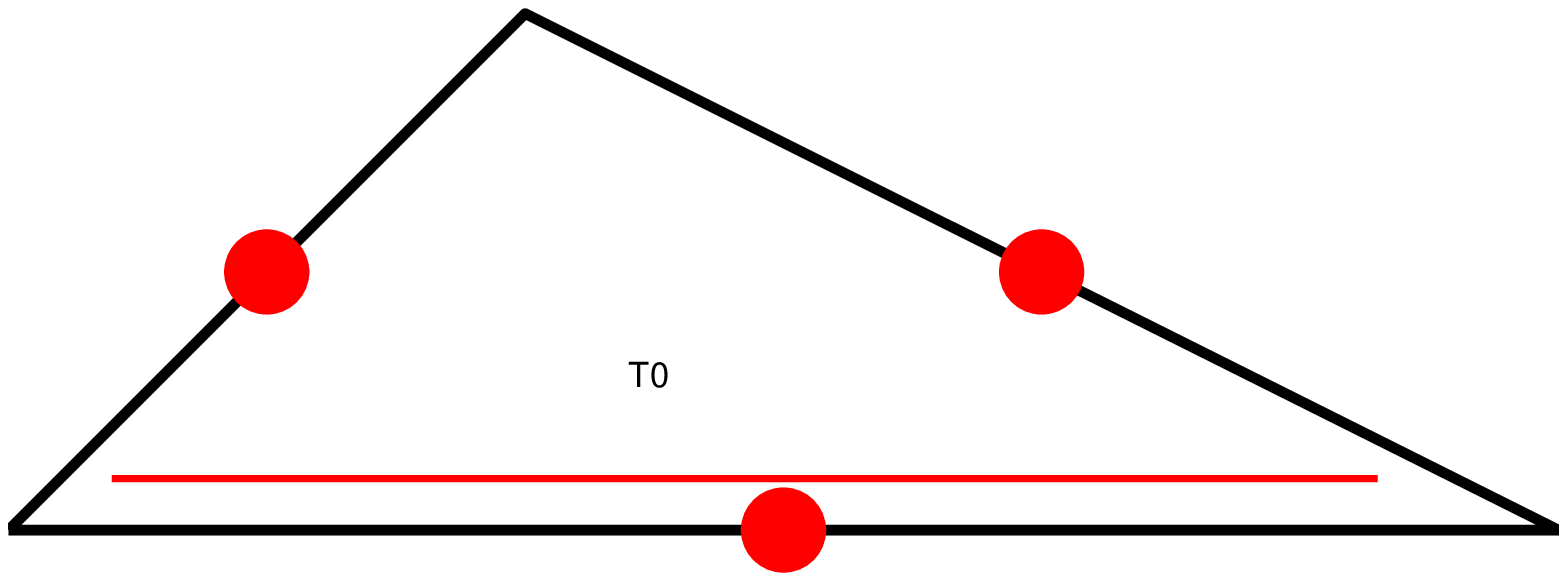} \\
 \includegraphics[width=35mm]{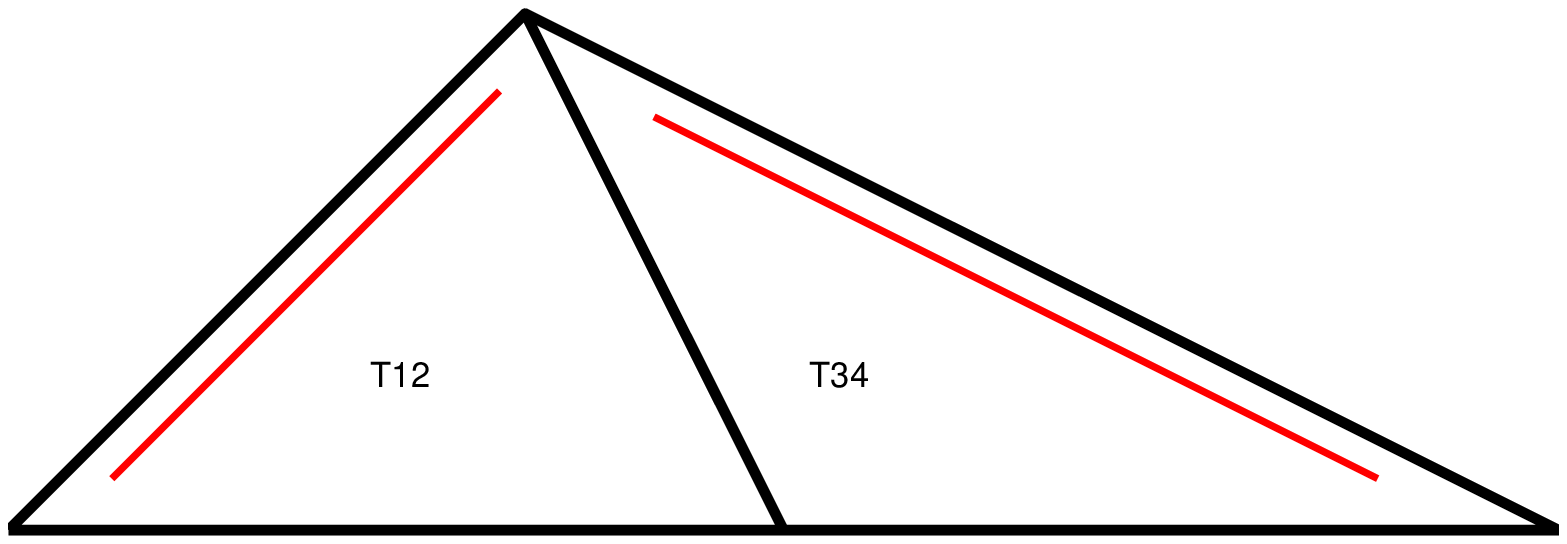} \quad
 \includegraphics[width=35mm]{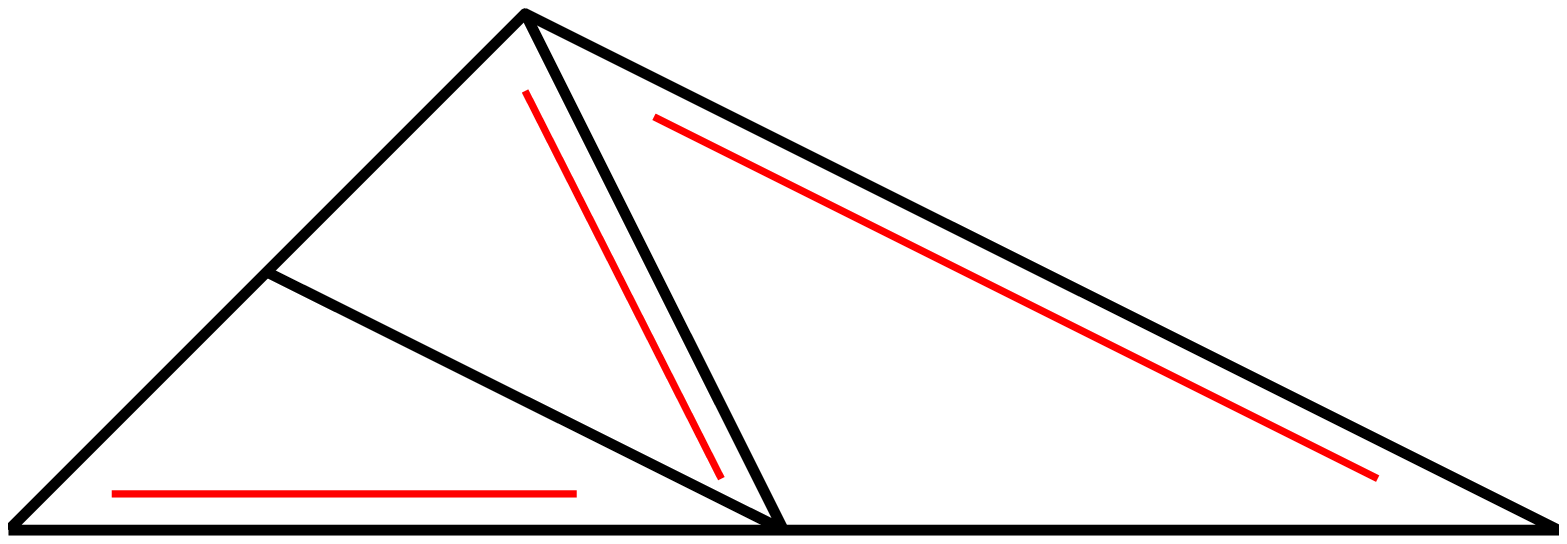}\quad
 \includegraphics[width=35mm]{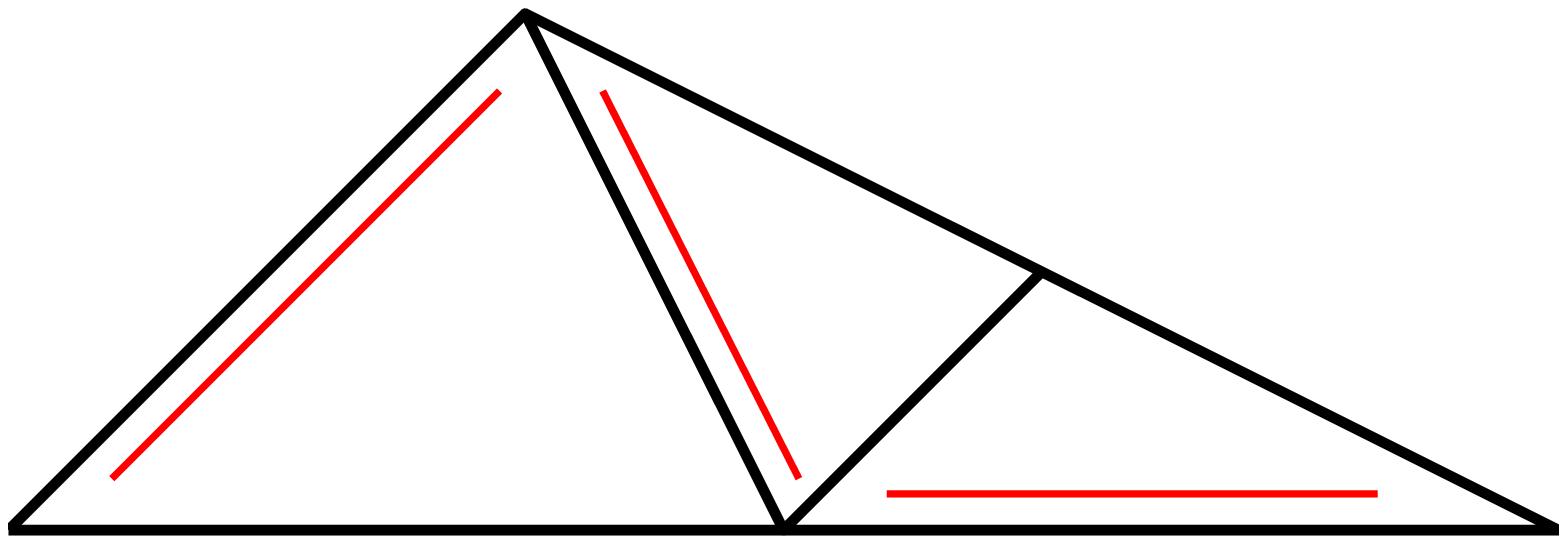}\quad
 \includegraphics[width=35mm]{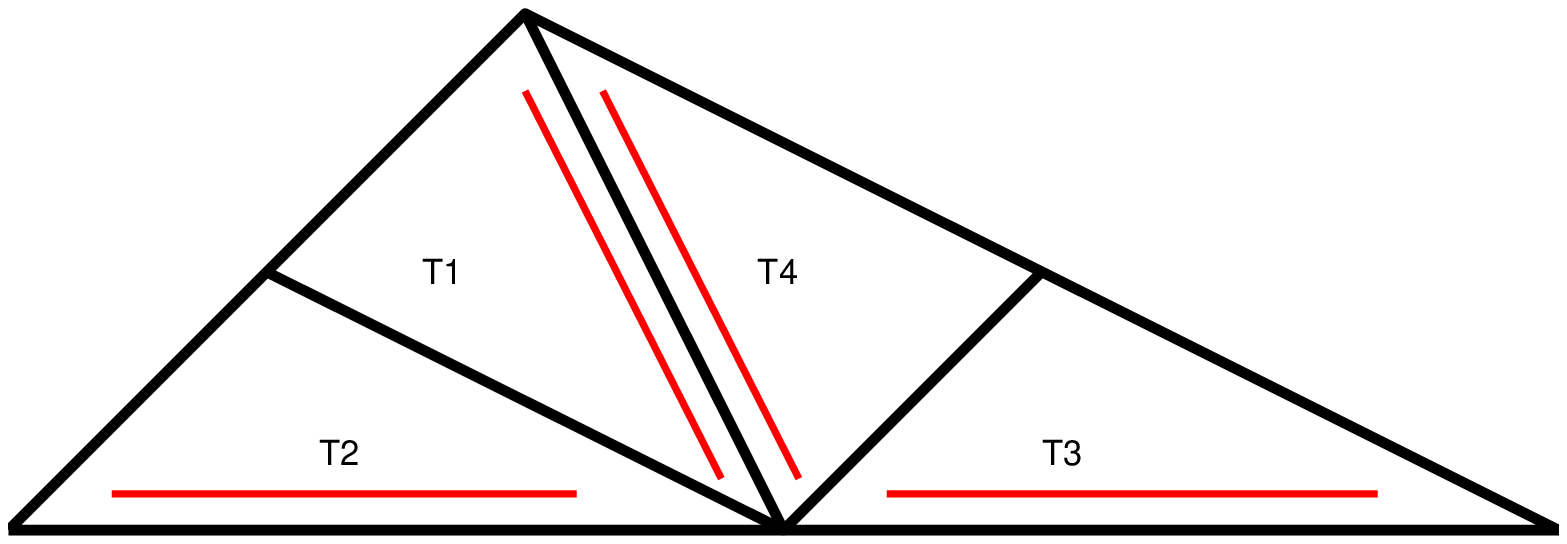}
 \caption{
 For each surface triangle $T\in\TT_\ell$ for $d=3$, there is one fixed 
 \emph{reference edge},
 indicated by the double line (left, top). Refinement of $T$ is done by bisecting
 the reference edge, where its midpoint becomes a new node. The reference
 edges of the son triangles are opposite to this newest vertex (left, bottom).
 To avoid hanging nodes, one proceeds as follows:
 We assume that certain edges of $T$, but at least the reference edge,
 are marked for refinement (top).
 Using iterated newest vertex bisection, the element is then split into
 2, 3, or 4 son triangles (bottom).}
 \label{fig:nvb:bisec}
\end{figure}

%--------------------------------------------------------------------------------
\subsection{Mesh refinement and hierarchical structure}\label{section:main:hierarchy}
%--------------------------------------------------------------------------------
\def\refine{{\tt refine}}
\def\MM{\mathcal M}
We assume that $\TT_\ell$ is obtained from an initial triangulation 
$\TT_0$ by use of bisection. For $d=2$, we employ the optimal 1D
bisection from~\cite{cmam} which guarantees $\ell$-independent $\gamma$-shape 
regularity~\eqref{eq:shaperegular}. For $d=3$, we use 2D newest vertex
bisection, see Figure~\ref{fig:nvb:bisec} as well as, e.g.,~\cite{kpp} and the references therein,
and note that $\ell$-independent $\gamma$-shape regularity~\eqref{eq:shaperegular}
is guaranteed.
We suppose that $\TT_{\ell+1} = \refine(\TT_{\ell};\MM_\ell)$ for all $\ell\in\N_0$, where $\refine(\cdot)$
abbreviates the mesh-refinement strategies mentioned and $\MM_\ell\subseteq\TT_\ell$ is an arbitrary set of marked elements. The mesh $\TT_{\ell+1}$ is then
the coarsest regular triangulation of $\Gamma$ such that all marked 
elements $T\in\MM_\ell$ have been bisected. 

Note that $\NN_\ell\subseteq\NN_{\ell+1}$, since $\TT_{\ell+1}$ is obtained by
local refinement of $\TT_\ell$.
To provide an efficient additive Schwarz scheme on locally refined meshes, 
we define
\begin{align}\label{eq:defNtilde}
  \widetilde\NN_0 = \NN_0
  \quad\text{and}\quad
  \widetilde\NN_\ell := \NN_\ell\backslash \NN_{\ell-1} \cup
  \set{z\in\NN_\ell\cap\NN_{\ell-1}}{\omega_\ell(z) \subsetneqq \omega_{\ell-1}(z)}
  \quad\text{for }\ell\ge1, 
\end{align}
i.e., $\widetilde\NN_\ell$  contains all \emph{new} nodes plus 
their immediate neighbours, see also Figure~\ref{fig:Ntilde}.
We stress that smoothing on all nodes $z\in\NN_\ell$ will lead to suboptimal
schemes, whereas smoothing on the nodes $z\in\widetilde\NN_\ell$ will prove to
be optimal.
For $\ell\ge0$ and $z\in\NN_\ell$, we define the subspaces 
\begin{align}\label{eq:defXtilde}
 \widetilde\XX^\ell &:= \linhull\set{\eta_z^\ell}{z\in\widetilde\NN_\ell}
 \quad\text{and}\quad
 \XX_z^\ell := \linhull\{\eta_z^\ell\}.
\end{align}

Finally, and as in~\cite{wuchen06}, we define the level of a node $z\in\NN_\ell$ by
\begin{align}\label{eq:level}
  \level_\ell(z) := \left\lfloor \frac{ \log(h_\ell(z) / \widehat h_0)}{\log(1/2)}
  \right\rfloor \in\N_0,
\end{align}
where $\widehat h_0 := \max\limits_{T\in\TT_0} h_\ell(T)$ and
$\lfloor\cdot\rfloor$ denotes the Gaussian floor function, i.e.,
$\lfloor x\rfloor = \max\set{n\in\N}{n\le x}$ for $x\in\R$.

\begin{figure}[t]
 \centering
 \includegraphics[width=65mm]{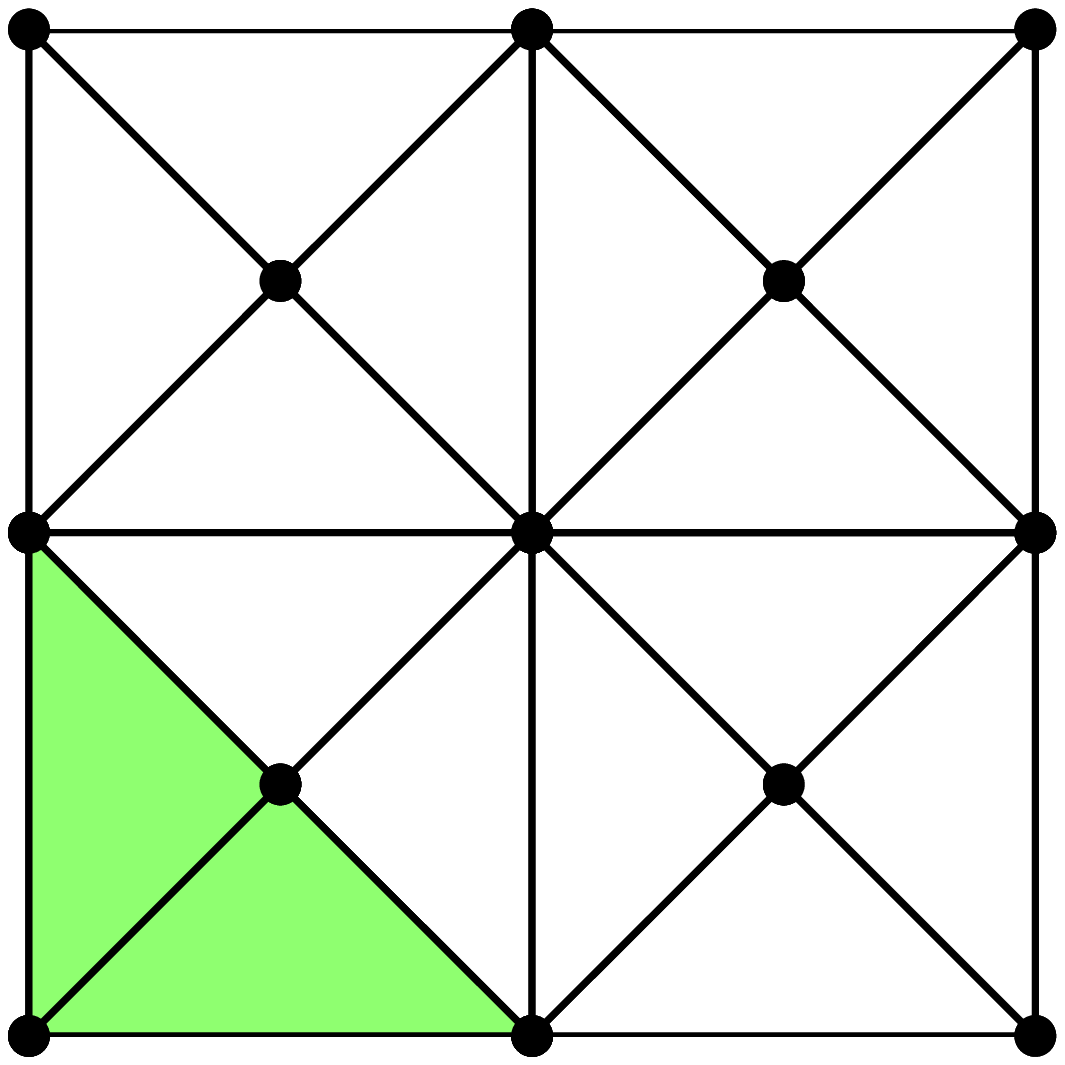} \quad
 \includegraphics[width=65mm]{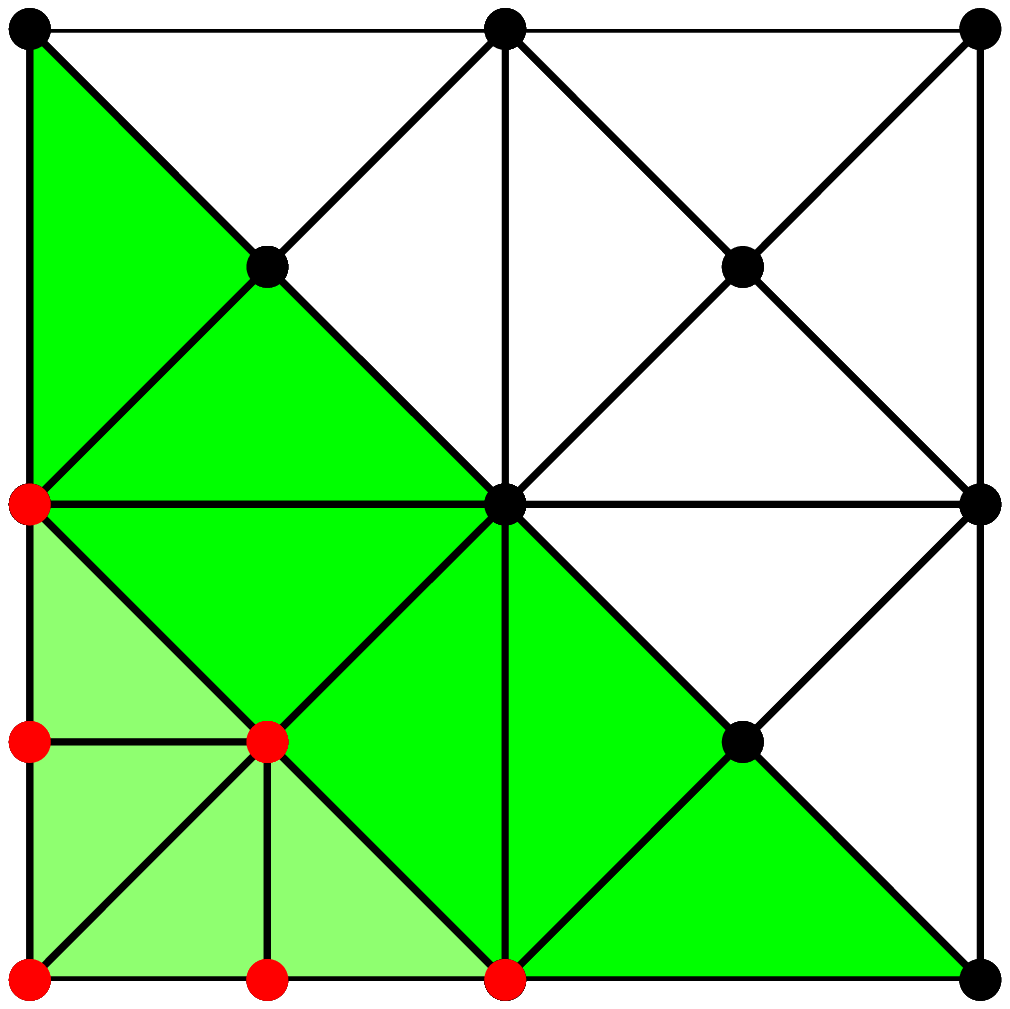} 
 \caption{
 The left figure shows a mesh $\TT_{\ell-1}$, where the two elements in the lower left corner
 are marked for refinement (green). Bisection of these two elements provides the mesh
 $\TT_\ell$ (right), where two \textit{new nodes} are created.
 The set $\widetilde\NN_\ell$ consists of these new nodes plus their 
 \textit{immediate neighbours} (red). 
 The union of the support of basis functions in
 $\widetilde\XX^\ell = \linhull\{\eta_z^\ell\,:\, z\in\widetilde\NN_\ell \}$
 is given by the light- and dark-green areas in the right figure.
 }
 \label{fig:Ntilde}
\end{figure}

%--------------------------------------------------------------------------------
\subsection{Local multilevel diagonal preconditioner (LMLD)}\label{section:main:precond}
%--------------------------------------------------------------------------------
For any $L\in\N_0$, we aim to derive a preconditioner
$(\widetilde\BB^L)^{-1}$ for the Galerkin matrix 
$\AA^L$ from~\eqref{eq:galerkinmatrix} with respect to the space 
$\XX^L$ and the basis $\set{\eta_z^L}{z\in\NN_L}$.

\def\II{\mathcal I}
For all $0\le \ell\le L$, let $\widetilde\AA^\ell$ be the Galerkin matrix 
with respect to $\widetilde\XX^\ell$ and the associated basis 
$\set{\eta_z^\ell}{z\in\widetilde\NN_\ell}$.
Let $\widetilde\DD^\ell$ be the diagonal matrix 
$\widetilde\DD^\ell_{jk}:=\widetilde\AA^\ell_{jj}\delta_{jk}$ with 
Kronecker's delta $\delta_{jk}$. 
Define $\widetilde N_\ell = \#\widetilde\NN_\ell$ and $N_L = \#\NN_L$. 
We consider the embedding $\widetilde\II^\ell:\widetilde\XX^\ell \to \XX^L$,
i.e.,\ the formal identity. 
Let $\widetilde{\bf I}^\ell\in \R^{N_L\times \widetilde N_\ell}$ be the
matrix representation of the operator $\widetilde\II^\ell$ with
respect to the bases of $\widetilde\XX^\ell$ resp.\ $\XX^L$. With this notation,
we consider the matrix
\begin{align}\label{eq:defBB}
  (\widetilde\BB^L)^{-1} = \sum_{\ell=0}^L 
  \widetilde{\bf I}^\ell\,(\widetilde\DD^\ell)^{-1}(\widetilde{\bf I}^\ell)^T.
\end{align}
Instead of solving  $\AA^L \xx^L = \bb^L$, we now consider the preconditioned linear system
\begin{align}\label{eq:precondSys}
  %\PPAS \xx^L := 
  (\widetilde\BB^L)^{-1}\AA^L \xx^L =
  (\widetilde\BB^L)^{-1}\bb^L.
\end{align}
As is shown in Section~\ref{sec:as}, $(\widetilde\BB^L)^{-1}$ 
corresponds to a diagonal scaling on each \textit{local} subspace $\widetilde\XX^\ell$.
Therefore, this type of preconditioner is called~\textit{local multilevel
diagonal scaling}.

For a symmetric and positive definite matrix $\CC \in \R^{N_L\times N_L}$, we denote by $\dual\cdot\cdot_\CC =
\dual{\CC\cdot}\cdot_2$ the induced scalar product on $\R^{N_L}$, and by $\norm\cdot\CC$ the
corresponding norm resp.\ induced matrix norm.
Here $\dual\cdot\cdot_2$ denotes the Euclidean inner product
on $\R^{N_L}$.
We define the condition number $\cond_\CC$ of a matrix $\AA\in\R^{N_L\times N_L}$ as
\begin{align}
  \cond_\CC(\AA) := \norm{\AA}\CC\norm{\AA^{-1}}\CC.
\end{align}

The main result of this work reads as follows.
\begin{theorem}\label{thm:main}
The matrix $(\widetilde\BB^L)^{-1}$ is symmetric and positive definite with respect to
$\dual\cdot\cdot_2$, and $\PPAStilde^L:
= (\widetilde\BB^L)^{-1}\AA^L$ is symmetric and positive definite with respect to 
$\dual\cdot\cdot_{\widetilde\BB^L}$.
Moreover, the minimal and maximal eigenvalues of the matrix
$\PPAStilde^L$ satisfy
\begin{align}
  c \le \evmin(\PPAStilde^L) \quad\text{and}\quad
  \evmax(\PPAStilde^L) \leq C,
\end{align}
where the constants $c,C>0$ depend only on $\Gamma$ and the initial
triangulation $\TT_0$.
In particular, the condition number of the additive Schwarz matrix
$\PPAStilde^L$ is
$L$-independently bounded by
\begin{align}
  \cond_{\widetilde\BB^L}(\PPAStilde^L) \leq C/c.
\end{align}
\end{theorem}

The proof of Theorem~\ref{thm:main} is given in Section~\ref{section:proof} below,
and we focus on the relevant application first:
Consider an iterative solution method, such as the GMRES method~\cite{saad}
or the CG method~\cite{saad03} to solve~\eqref{eq:precondSys}, where
the relative reduction of the $j$-th residual depends only on the condition number
$\cond_{\widetilde\BB}(\PPAStilde^L)$. Then, Theorem~\ref{thm:main} proves that the iterative
scheme, together with the preconditioner $\widetilde\BB^L$ is efficient in the sense
that the number of iterations to reduce the relative residual under the
tolerance $\eps$ is bounded by a constant that depends only on $\Gamma$, and the initial triangulation
$\TT_0$, but is completely independent of the current triangulation $\TT_L$.

\subsection{Global multilevel diagonal preconditioner (GMLD)}\label{section:main:gmld}
In addition to the new local multilevel diagonal preconditioner from
Section~\ref{section:main:precond}, we consider a \textit{global multilevel diagonal} 
preconditioner, where we use all nodes $z\in\NN_\ell$ of the
triangulation $\TT_\ell$ instead of only $\widetilde\NN_\ell$ to construct the
preconditioner.
Such an approach has, for instance, been investigated in~\cite{mai09} for 2D
hypersingular integral equations with graded meshes
on an open curve. In the latter work it is proved that the maximal eigenvalue
of the associated additive Schwarz operator is bounded up to some constant by
$L^2$.
With the new analytical tools developed here, we improve this estimate and show that the maximal
eigenvalue can be bounded linearly in $L$. Our result holds for all $\gamma$-shape
regular meshes on open and closed boundaries for $d=2,3$.
Moreover, this bound is sharp, as it is confirmed in the numerical example from
Section~\ref{sec:artLshape}.

Let $\AA^\ell$ denote the Galerkin matrix with respect to $\XX^\ell$
and the basis $\{\eta_z^\ell\,:\, z\in\NN_\ell\}$. We consider the diagonal
$\DD^\ell$ of this matrix, i.e. $\DD_{jk}^\ell := \delta_{jk} \AA_{jj}^\ell$.
Let $\II^\ell : \XX^\ell \to \XX^L$ denote the canonical embedding and let
$\mathbf{I}^\ell$ denote its matrix representation with respect to the nodal
basis of $\XX^\ell$ and $\XX^L$.
We define the global multilevel preconditioner $(\BB^L)^{-1}$ by
\begin{align}
  (\BB^L)^{-1} := \sum_{\ell=0}^L  \mathbf{I}^\ell
  (\DD^\ell)^{-1} (\mathbf{I}^\ell)^T.
\end{align}
\begin{theorem}\label{thm:gmld}
The matrix $(\BB^L)^{-1}$ is symmetric and positive definite with respect to
$\dual\cdot\cdot_2$, and $\PPAS^L:
= (\BB^L)^{-1}\AA^L$ is symmetric and positive definite with respect to $\dual\cdot\cdot_{\BB^L}$.
Moreover, the minimal and maximal eigenvalues of the matrix $\PPAS^L$ satisfy
\begin{align}
 c \le \evmin(\PPAS^L) \quad\text{and}\quad \evmax(\PPAS^L) \leq C (L+1),
\end{align}
where the constants $c,C>0$ depend only on $\Gamma$ and the initial
triangulation $\TT_0$, but are independent of the level $L$.
In particular, the condition number of the additive Schwarz matrix $\PPAS^L$ is
bounded by
\begin{align}
  \cond_{\BB^L}(\PPAS^L) \leq (L+1)\, C/c.
\end{align}
\end{theorem}

% !TEX root = hypsingAS.tex

%--------------------------------------------------------------------------------
\section{Numerical examples}\label{sec:examples}
%--------------------------------------------------------------------------------
\noindent
In this section we consider different 2D experiments to show
the efficiency of the proposed \textit{local multilevel diagonal
preconditioner} $\widetilde\BB^L$ from~\eqref{eq:precondSys} numerically. In the 
first two experiments, we consider problems on the boundary of an L-shaped domain.
In the last experiment, we consider a problem on the slit $\Gamma =
(-1,1)\times \{0\}$.
In both cases, the exact solution is known. Moreover, for the problems under
considerations it is known, that uniform mesh-refinement will lead to suboptimal 
convergence rates, whereas adaptive refinement regains the optimal order of
convergence. Thus, the use of adaptive methods is preferable. 

For both examples, the mesh-adaptivity is driven by the ZZ-type error estimator 
proposed in~\cite{zzbem}. The resulting linear systems are solved by GMRES. We 
compare the following preconditioners  with respect to time, number of iterations, 
and condition numbers:
\begin{itemize}
\item \textbf{LMLD} \textit{local multilevel diagonal preconditioner} from~\eqref{eq:precondSys};
\item \textbf{GMLD} \textit{global multilevel diagonal preconditioner}
  from Section~\ref{section:main:gmld};
\item \textbf{HB} \textit{hierarchical basis preconditioner}, where only new 
  nodes are considered for preconditioning~\cite{tsm97}: Define $\underline\NN_\ell :=
\NN_\ell\backslash\NN_{\ell-1}$ as the set of new nodes and define
$\underline\DD^\ell$ as the diagonal matrix of the Galerkin matrix with respect
to the space $\underline\XX^\ell := \{\eta_z^\ell \,:\, z\in\underline\NN_\ell\}$.
Let $\underline\II^\ell : \underline\XX^\ell\to \XX^L$ denote the canonical
embedding with matrix representation $\underline{\mathbf I}^\ell$. The
hierarchical basis preconditioner is then given by
\begin{align}
  \BB_{\rm HB}^L := \sum_{\ell=0}^L \underline{\mathbf I}^\ell
  (\underline\DD^\ell)^{-1} (\underline{\mathbf I}^\ell)^T.
\end{align}
The preconditioned matrix reads $\PP_{\rm HB}^L := (\BB_{\rm HB}^L)^{-1}
\AA^L$.

\item \textbf{DIAG} \textit{diagonal scaling} of the Galerkin
  matrix~\cite{amt99,gm06}. The preconditioned matrix reads $\PP_{\rm diag}^L := (\DD^L)^{-1} \AA^L$.
\end{itemize}

%%%%%%%%%%%%%%%%%%%%%
% FIGURES
% %%%%%%%%%%%%%%%%%%%%
\begin{figure}[th]
  \begin{center}
    \psfrag{x}{\tiny $x$}
    \psfrag{y}{\tiny $y$}
  \includegraphics[width=0.6\textwidth]{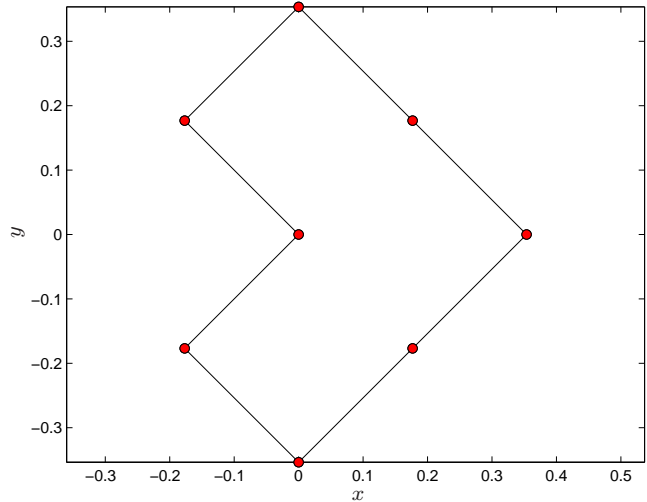}
  \caption{L-shaped domain $\Omega$ and initial triangulation $\TT_0$ with
  $\#\TT_0 = 8$.}
  \label{fig:Lshape}
\end{center}
\end{figure}
\begin{figure}[th]
  \begin{center}
    \psfrag{number of elements}[c][c]{\tiny number of elements $N$}
    \psfrag{condition numbers}[c][c]{\tiny condition numbers}
    \psfrag{OS}[0][-145]{\tiny $\OO(N \log(N\hmin{L})$}
    \psfrag{O1}{\tiny $\OO(1)$}
  \includegraphics[width=0.6\textwidth]{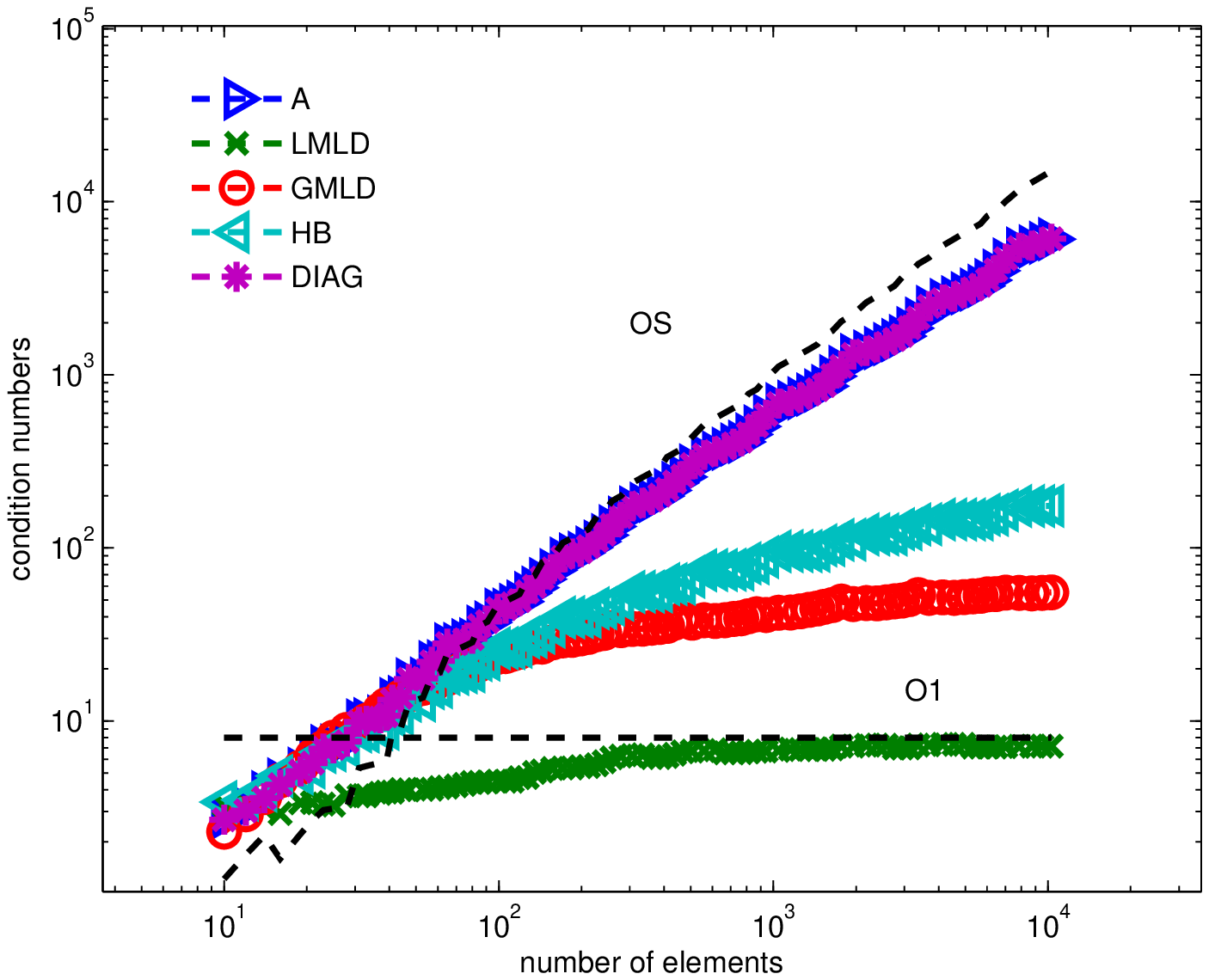}
  \caption{Condition numbers $\evmax/\evmin$ of the preconditioned
  Galerkin matrices $\PPAStilde^L$ (LMLD), $\PPAS^L$ (GMLD), $\PP_{\rm HB}^L$ (HB), 
  $\PP_{\rm diag}^L$ (DIAG), and the unpreconditioned matrix $\AA^L$ (A) for the problem from
  Section~\ref{sec:Lshape}.}
  \label{fig:cond_Lshape}
\end{center}
\end{figure}
\begin{figure}[th]
\begin{center}
    \psfrag{number of elements}[c][c]{\tiny number of elements $N$}
    \psfrag{time}[c][c]{\tiny time [sec.]}
    \psfrag{ON}{\tiny $\OO(N)$}
    \psfrag{ON2}{\tiny $\OO(N^2)$}
  \includegraphics[width=0.6\textwidth]{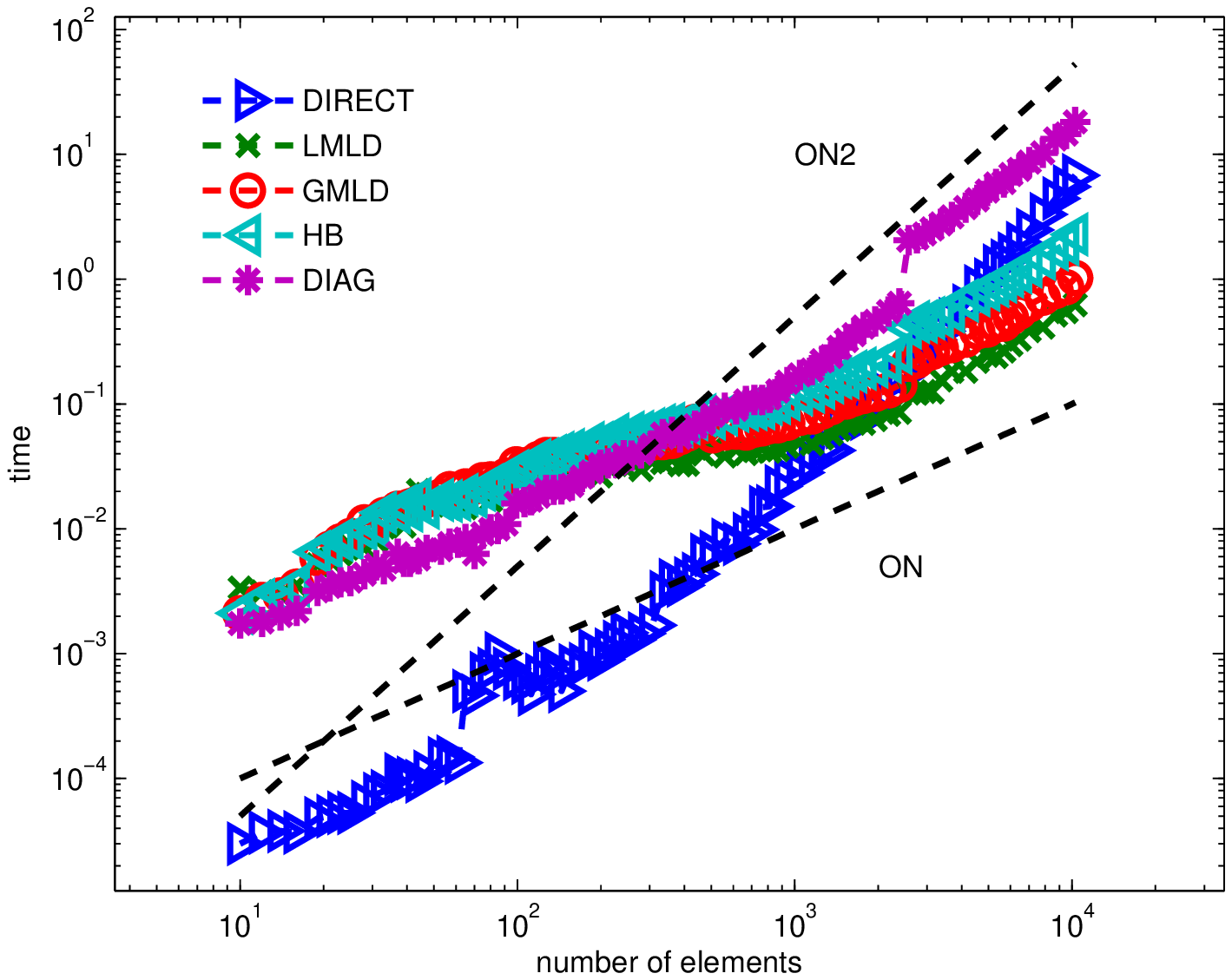}
  \caption{Computational times for GMRES algorithm to reduce relative residual
  under the bound $\eps = 10^{-8}$ for the problem from
  Section~\ref{sec:Lshape}
  and the preconditioned Galerkin systems $\PPAStilde^L$ (LMLD), $\PPAS^L$ (GMLD), $\PP_{\rm HB}^L$ (HB), 
  $\PP_{\rm diag}^L$ (DIAG).
  DIRECT stands for the direct solver.
  }
  \label{fig:time_Lshape}
\end{center}
\end{figure}
\begin{figure}[th]
\begin{center}
    \psfrag{number of elements}[c][c]{\tiny number of elements $N$}
    \psfrag{iterations}[c][c]{\tiny number of iterations}
  \includegraphics[width=0.6\textwidth]{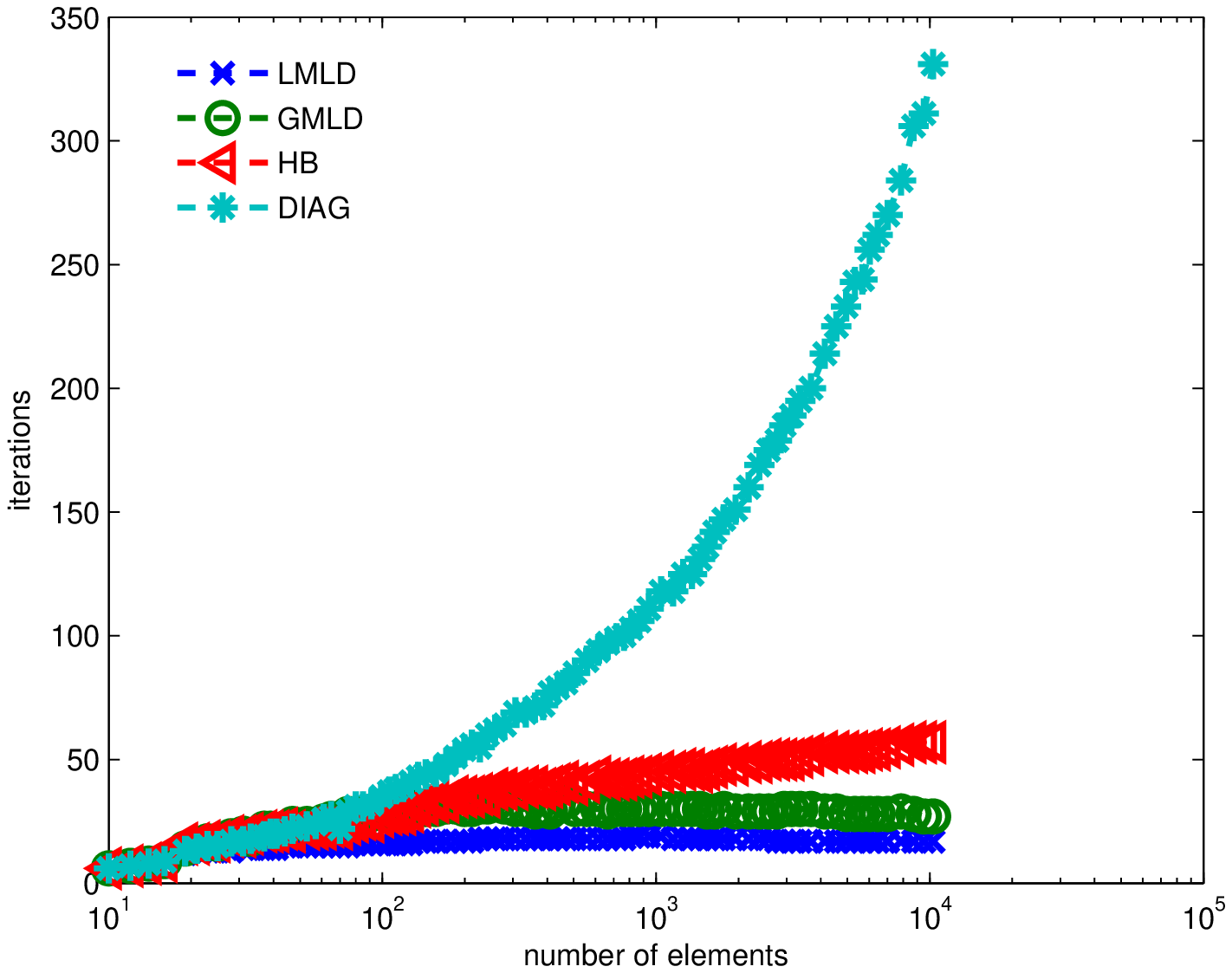}
  \caption{Number of iterations of GMRES to reduce relative residual
  under the bound $\eps = 10^{-8}$ for the problem from
  Section~\ref{sec:Lshape}
  and the preconditioned Galerkin systems $\PPAStilde^L$ (LMLD), $\PPAS^L$ (GMLD), $\PP_{\rm HB}^L$ (HB), 
  $\PP_{\rm diag}^L$ (DIAG).
  }
  \label{fig:iter_Lshape}
\end{center}
\end{figure}

%%%%%%%%%%%%%%%%%%%%%%%%%%%%%%%%%%%%%%%%%%%%%%%%%%%%%%%%%%%%%%%%%%%%%%%%%%%%%%%%%

\begin{figure}[th]
  \begin{center}
    \psfrag{number of elements}[c][c]{\tiny number of elements $N$}
    \psfrag{condition numbers}[c][c]{\tiny condition numbers}
    \psfrag{OS}{\tiny $\OO(N\log(N\hmin{L}))$}
    \psfrag{O1}{\tiny $\OO(1)$}
    \psfrag{OL}{\tiny $\OO(L)$}
    \psfrag{OL2}{\tiny $\OO(L^2)$}
  \includegraphics[width=0.6\textwidth]{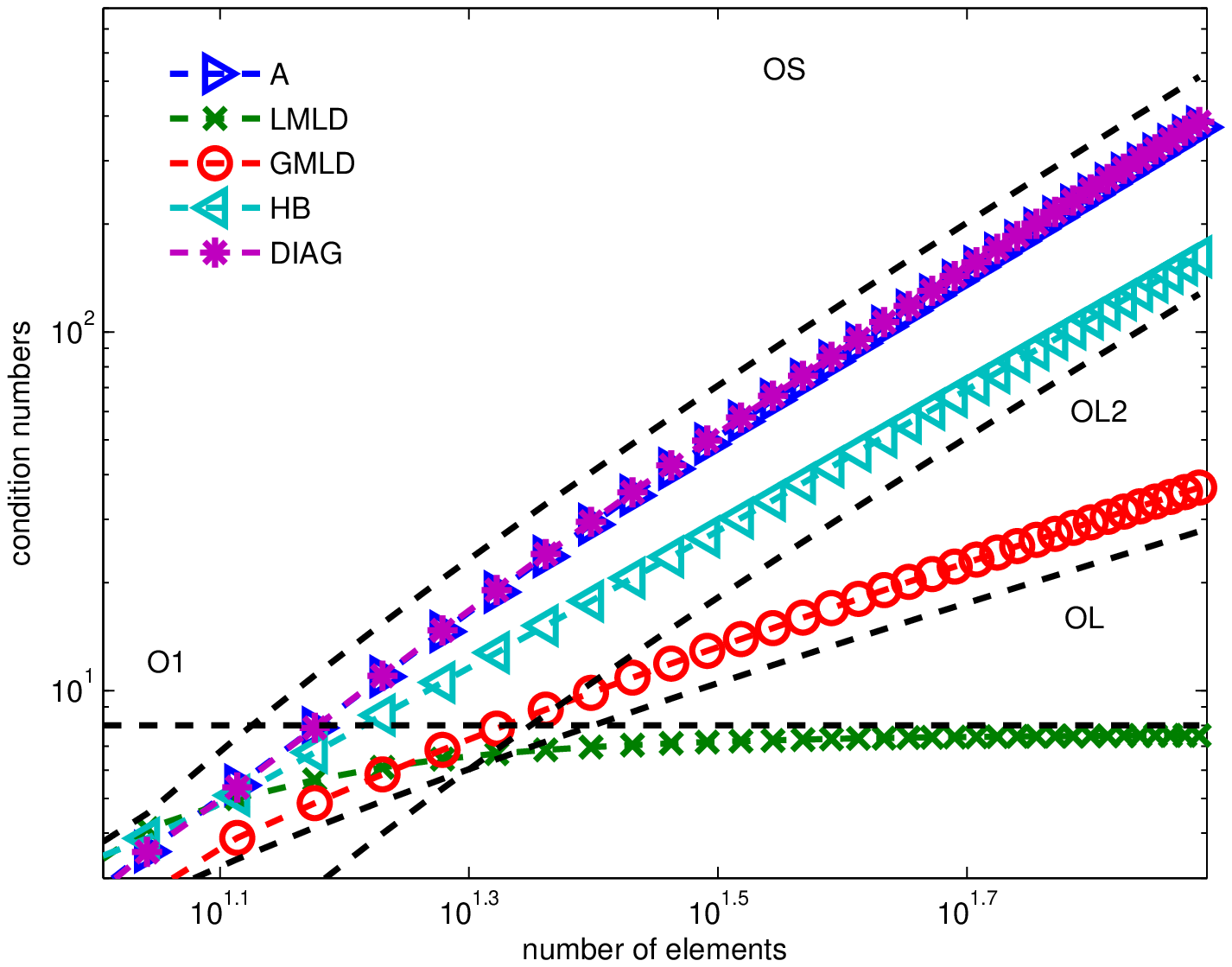}
  \caption{Condition numbers $\evmax/\evmin$ of the preconditioned
  Galerkin matrices $\PPAStilde^L$ (LMLD), $\PPAS^L$ (GMLD), $\PP_{\rm HB}^L$ (HB), 
  $\PP_{\rm diag}^L$ (DIAG), and the unpreconditioned matrix $\AA^L$ (A) for the problem from
  Section~\ref{sec:artLshape}.}
  \label{fig:cond_artLshape}
\end{center}
\end{figure}

%%%%%%%%%%%%%%%%%%%%%%%%%%%%%%%%%%%%%%%%%%%%%%%%%%%%%%%%%%%%%%%%%%%%%%%%%%%%%%%%%

\begin{figure}[th]
  \begin{center}
    \psfrag{number of elements}[c][c]{\tiny number of elements $N$}
    \psfrag{condition numbers}[c][c]{\tiny condition numbers}
    \psfrag{OS}[0][-145]{\tiny $\OO(N \log(N\hmin{L}))$}
    \psfrag{O1}{\tiny $\OO(1)$}
  \includegraphics[width=0.6\textwidth]{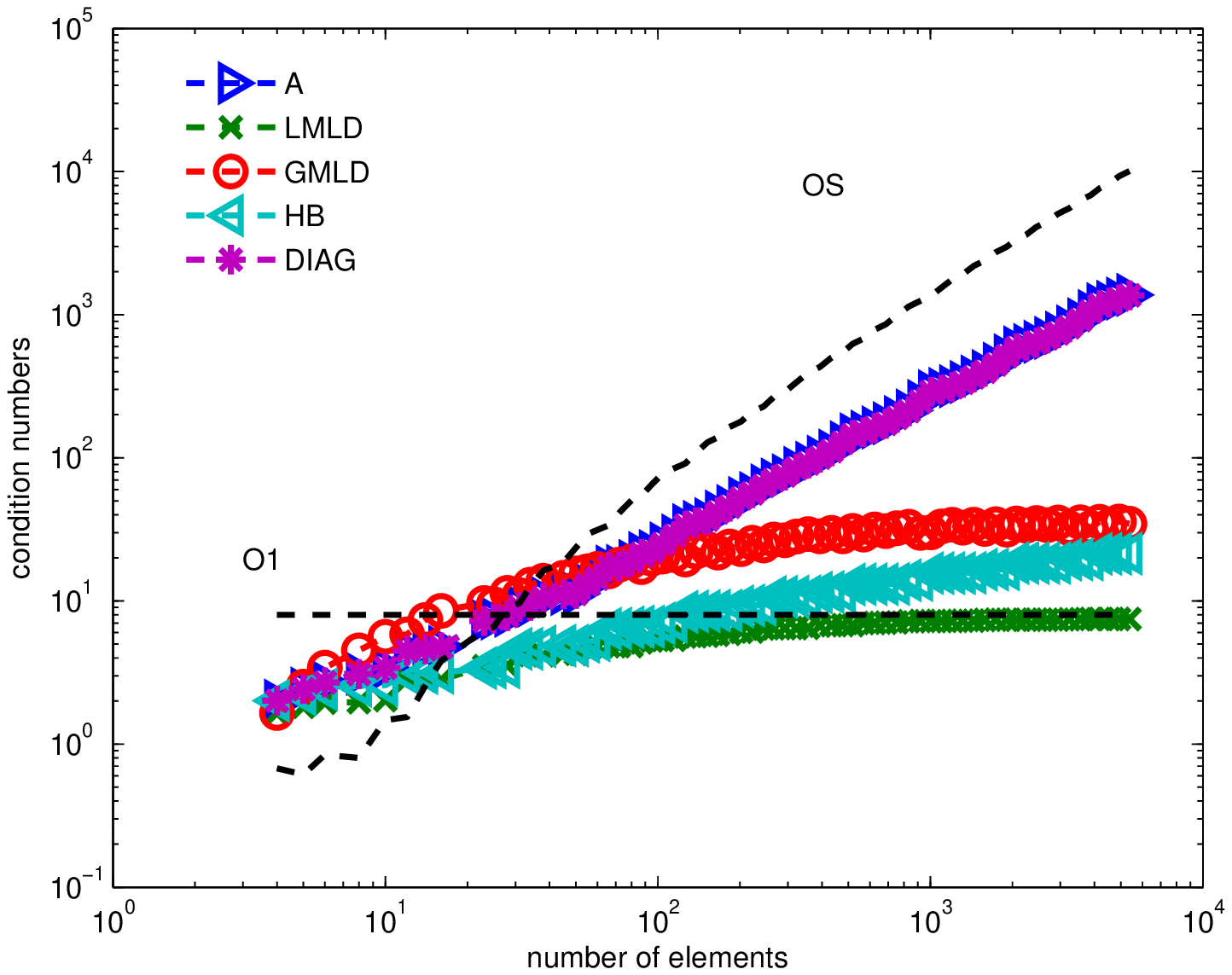}
  \caption{Condition numbers $\evmax/\evmin$ of the preconditioned
  Galerkin matrices $\PPAStilde^L$ (LMLD), $\PPAS^L$ (GMLD), $\PP_{\rm HB}^L$ (HB), 
  $\PP_{\rm diag}^L$ (DIAG), and the unpreconditioned matrix $\AA^L$ (A) for the problem from
  Section~\ref{sec:slit}.}
  \label{fig:cond_slit}
\end{center}
\end{figure}
\begin{figure}[th]
\begin{center}
    \psfrag{number of elements}[c][c]{\tiny number of elements $N$}
    \psfrag{time}[c][c]{\tiny time [sec.]}
    \psfrag{ON}{\tiny $\OO(N)$}
    \psfrag{ON2}{\tiny $\OO(N^2)$}
  \includegraphics[width=0.6\textwidth]{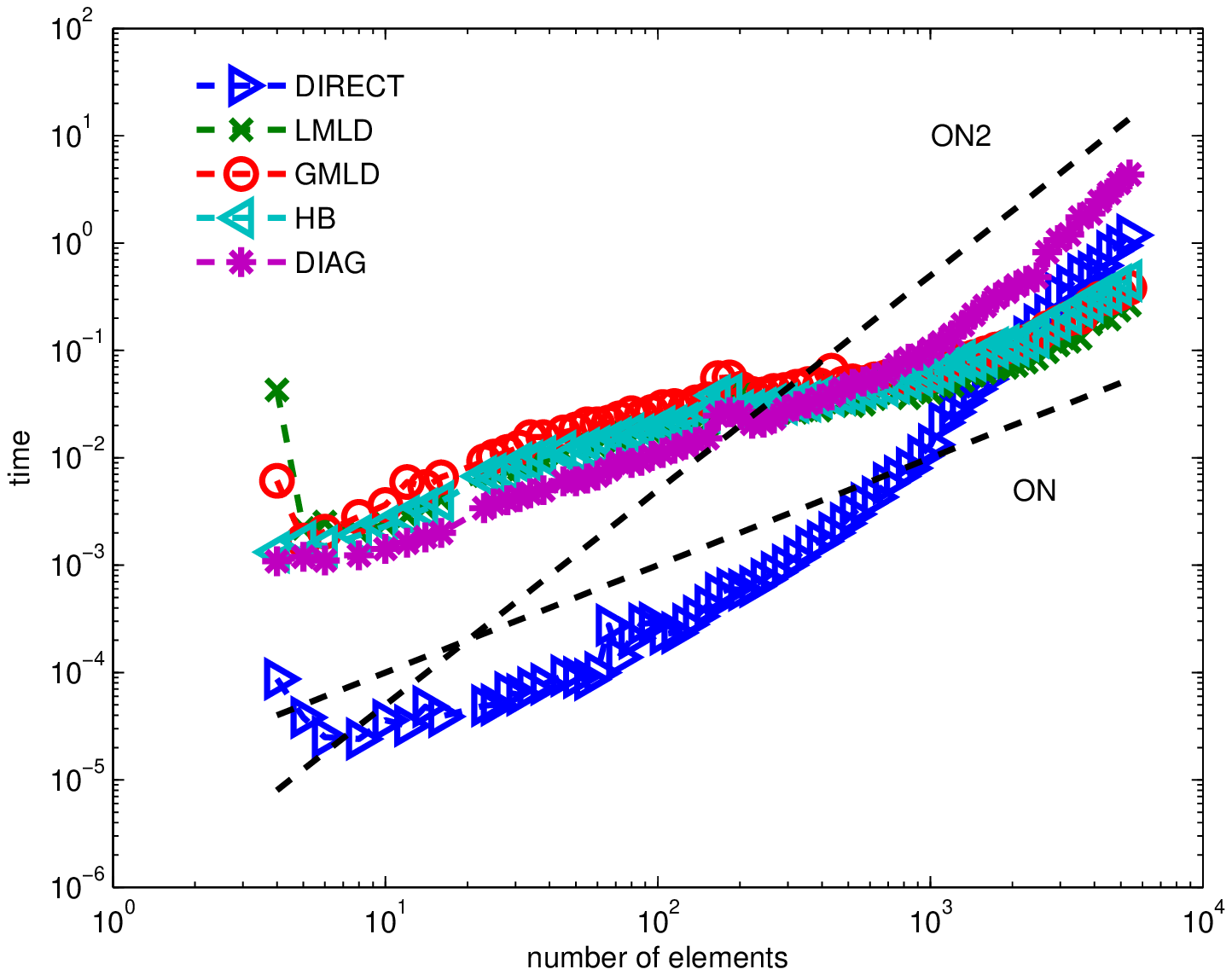}
  \caption{Computational times for GMRES algorithm to reduce relative residual
  under the bound $\eps = 10^{-8}$ for the problem from
  Section~\ref{sec:slit}
  and the preconditioned Galerkin systems $\PPAStilde^L$ (LMLD), $\PPAS^L$ (GMLD), $\PP_{\rm HB}^L$ (HB), 
  $\PP_{\rm diag}^L$ (DIAG).
  DIRECT stands for the direct solver.
  }
  \label{fig:time_slit}
\end{center}
\end{figure}
\begin{figure}[th]
\begin{center}
    \psfrag{number of elements}[c][c]{\tiny number of elements $N$}
    \psfrag{iterations}[c][c]{\tiny number of iterations}
  \includegraphics[width=0.6\textwidth]{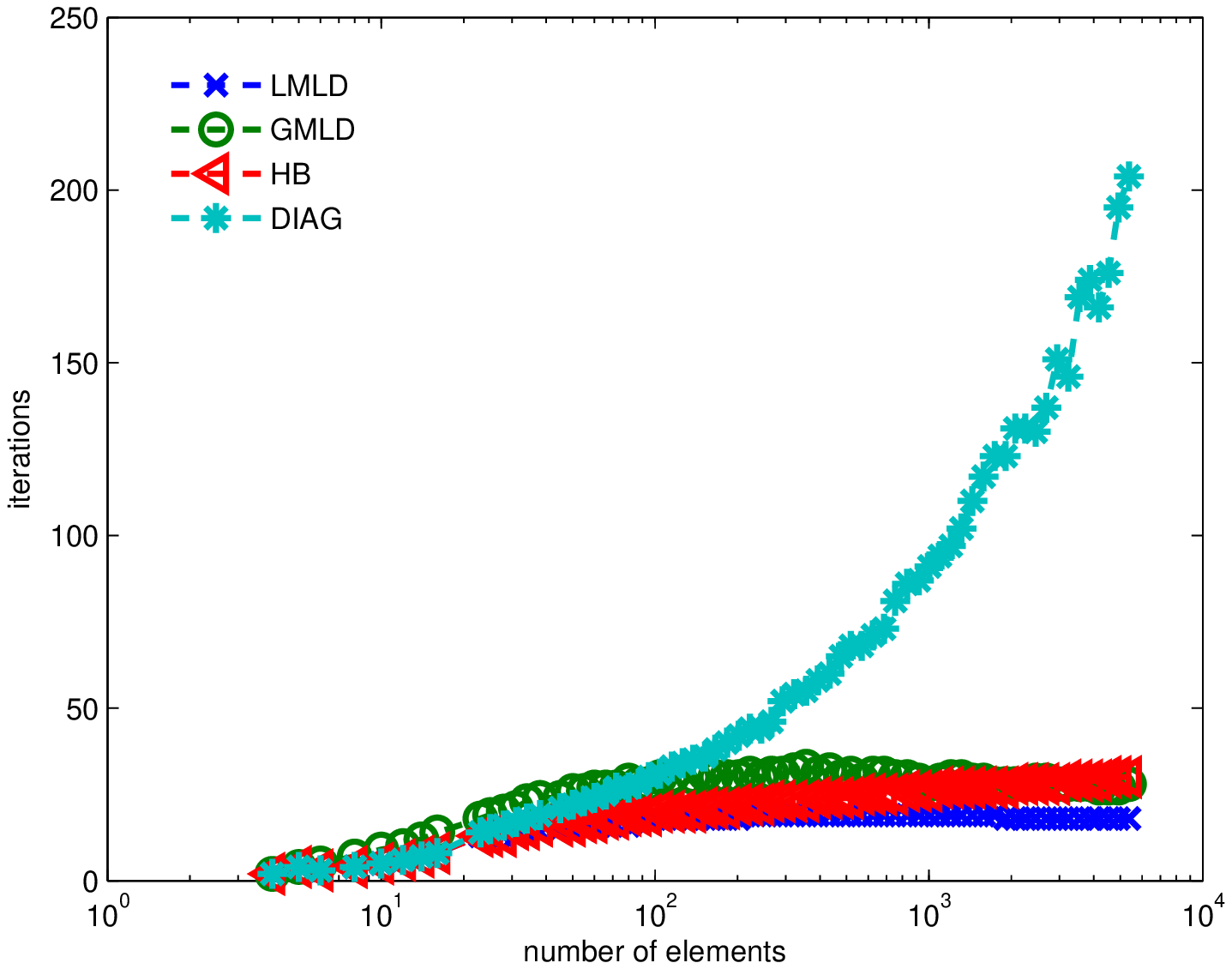}
  \caption{Number of iterations of GMRES to reduce relative residual
  under the bound $\eps = 10^{-8}$ for the problem from
  Section~\ref{sec:slit}
  and the preconditioned Galerkin systems $\PPAStilde^L$ (LMLD), $\PPAS^L$ (GMLD), $\PP_{\rm HB}^L$ (HB), 
  $\PP_{\rm diag}^L$ (DIAG).
  }
  \label{fig:iter_slit}
\end{center}
\end{figure}

The preconditioners GMLD and HB are very similar to LMLD, but lack 
effectivity, i.e., the condition number depends on the level $L$ resp.\ on the 
mesh-width function $h_L$. Since the diagonal elements of the matrix $\AA^L$ 
are essentially constant, the simple diagonal preconditioner has no significant 
effect on the condition numbers, see also~\cite{amt99}. Moreover,
the GMRES-based iterative solution is compared to a direct solver
\textbf{DIRECT} for the unpreconditioned system.

All computations were performed with MATLAB (2012b) on an Intel Core
i7-3930K machine with 32GB RAM under a x86\_64 GNU/Linux system.
For the GMRES algorithm, we use the MATLAB function {\tt gmres.m}.
For the direct solver, we use the MATLAB backslash operator.
The assembly of the boundary integral operators was done with
help of the MATLAB BEM-library \textsc{HILBERT}~\cite{hilbert}.

%--------------------------------------------------------------------------------
\subsection{Adaptive BEM for hypersingular integral equation for 2D Neumann problem on L-shaped domain}\label{sec:Lshape}
%--------------------------------------------------------------------------------
We consider the boundary $\Gamma = \partial\Omega$ of the L-shaped domain
$\Omega$, sketched in Figure~\ref{fig:Lshape}. With the 2D polar coordinates 
$(r,\varphi)$ of $x\in\R^2\backslash\{0\}$, the function
\begin{align*}
  w(x,y) = r^{2/3}\cos(2/3\varphi),
\end{align*}
satisfies the Neumann problem
\begin{align*}
 -\Delta w = 0\text{ in }\Omega,
 \quad\text{and}\quad
 \partial_n w =: \phi\text{ on }\Gamma.
\end{align*}
With the adjoint double-layer potential, we define the
right-hand side $f$ of~\eqref{eq:hypsing} as
\begin{align*}
  f = (1/2-\dlp')\partial_{\normal}w.
\end{align*}
The exact solution $u$ of~\eqref{eq:hypsing} is, up to some additive constant, 
the trace $u=w|_\Gamma$ of the potential $w$. We use the adaptive lowest-order
BEM from~\cite{zzbem} to approximate $u$. This leads to successively refined
meshes $\TT_\ell$ for $\ell=0,1,2\dots$.

In Figure~\ref{fig:cond_Lshape}, we compare the condition numbers of the different preconditioned matrices and the Galerkin matrix $\AA^L$, i.e., the
ratio between maximal and minimal eigenvalue.
As predicted by Theorem~\ref{thm:main}, the condition number of the
preconditioned matrix $\PPAStilde^L$ stays bounded, whereas the preconditioned 
systems which use GMLD resp.\ HB, are suboptimal. The condition numbers
of the unpreconditioned system DIRECT as well as the diagonally 
preconditioned system DIAG essentially coincide.

If we compare the computational times used for solving the different
preconditioned systems, we infer that the system with the matrix $\PPAStilde^L$ is
solved fastest, see Figure~\ref{fig:time_Lshape}. Clearly, this is because the number of iterations,
is the smallest of all iterative methods, cf.~Figure~\ref{fig:iter_Lshape}. We
obtain from Figure~\ref{fig:time_Lshape} that the computational time is almost
linear in the number of elements in the mesh $\TT_L$. However, direct solvers
are known to have higher computational costs. Moreover, the memory consumption
of direct solvers also exceeds that of an efficient iterative solution method.
Therefore, if the number of elements are large, efficient iterative methods are
inevitable.

%--------------------------------------------------------------------------------
\subsection{Artificial mesh-adaptation for hypersingular integral equation for 2D Neumann problem on L-shaped domain}\label{sec:artLshape}
%--------------------------------------------------------------------------------
We consider again the problem from the previous Section~\ref{sec:Lshape}.
However, we use a stronger mesh-adaptation towards the reentrant corner: We start with the initial mesh $\TT_0$ as given in Figure~\ref{fig:Lshape}, and obtain 
$\TT_\ell$ from $\TT_{\ell-1}$ by bisecting only the two elements closest to 
the origin $(0,0)$. Note that this refinement preserves $\gamma$-shape regularity 
of the triangulations $\TT_\ell$. Simple calculations show 
$\hmin\ell = 2^{-\ell} h_0$ and $\hmax\ell = h_0$, where $h_0\in\R$ denotes the 
constant mesh-width of the initial triangulation.
We compare the condition numbers of the unpreconditioned and preconditioned
systems in Figure~\ref{fig:cond_artLshape}. Here, we consider LMLD, GMLD,
and HB, as well as the simple diagonal scaling DIAG.
The results from Figure~\ref{fig:cond_artLshape} show that the condition number
of $\PPAStilde^L$ is uniformly bounded, while the condition number of $\PPAS^L$
grows with $L$. In particular, this proves that Theorem~\ref{thm:main}
and~\ref{thm:gmld} are sharp. The condition number of $\PP_{\rm HB}^L$ grows
even worse.
It is proved in~\cite[Corollary~1]{tsm97} that the condition number of HB 
is bounded by $|\log(\hmin{L})|^2 \simeq L^2$.
%--------------------------------------------------------------------------------
\subsection{Problem on slit}\label{sec:slit}
%--------------------------------------------------------------------------------
We consider the hypersingular equation~\eqref{eq:hypsing} on the slit $\Gamma =
(-1,1)\times\{0\}$ with right-hand side $f=1$ and exact solution $u(x,0) =
2\sqrt{1-x^2}$. For this example, the correct energy space is $\widetilde H^{1/2}(\Gamma)$,
and the exact solution belongs to $u\in(\widetilde H^{1/2}(\Gamma)\cap H^{1-\eps}(\Gamma))
\backslash H^1(\Gamma)$ for all $\eps>0$. In particular, uniform mesh-refinement
thus leads to the reduced order of convergence $\OO(h^{1/2})=\OO(N^{-1/2})$, while
the adaptive mesh-refinement of~\cite{zzbem} regains the optimal order
$\OO(N^{-3/2})$. 

As is shown in Section~\ref{sec:screen} below, the main result of 
Theorem~\ref{thm:main} also holds in this setting.
We compare different types of multilevel additive Schwarz preconditioners
like LMLD, GMLD, and HB, as well as simple diagonal scaling with
respect to condition numbers in Figure~\ref{fig:cond_slit} resp.\ 
the number of GMRES iterations in Figure~\ref{fig:iter_slit}. Moreover, 
we also compare GMRES vs.\ the direct solver with respect
to the computational time, see Figure~\ref{fig:time_slit}.

\section{Proof of Theorem~\ref{thm:main}}
\label{section:proof}

% !TEX root = hypsingAS.tex

%--------------------------------------------------------------------------------
\subsection{Abstract analysis of additive Schwarz operators}\label{sec:as}
%--------------------------------------------------------------------------------
In this subsection, we show that the multilevel diagonal scaling from
Section~\ref{sec:main} is a multilevel additive Schwarz method. Recall the
notation from Section~\ref{section:main:hierarchy} and
Section~\ref{section:main:precond}.

For each subspace $\XX_z^\ell = \linhull\{\eta_z^\ell\}$, we define the
projection $\prec_z^\ell : \XX^L\to \XX_z^\ell$ by
\begin{align}\label{eq:defproj}
  \edual{\prec_z^\ell v}{w_z^\ell} = \edual{v}{w_z^\ell} \quad\text{for all }
  w_z^\ell \in\XX_z^\ell.
\end{align}
Note that $\prec_z^\ell$ is the orthogonal projection onto
the one-dimensional space $\XX_z^\ell$.
Thus, the explicit representation of $\prec_z^\ell$ reads
\begin{align}\label{eq:Pz}
  \prec_z^\ell v =  \frac{\edual{v}{\eta_z^\ell}}{\enorm{\eta_z^\ell}^2}\,
  \eta_z^\ell.
\end{align}
Based on~$\prec_z^\ell$, we define
the additive Schwarz operator 
\begin{align}
  \PAStilde^L  := \sum_{\ell=0}^L\sum_{z\in\widetilde\NN_\ell}
 \prec_z^\ell : \XX^L \to \XX^L.
\end{align}
Let $v = \sum_{j=1}^{N_L} \xx_j \eta_{z_j}^L$, $w=
\sum_{k=1}^{\widetilde N_\ell} \yy_k \eta_{z_k}^\ell$.
A moment's reflection reveals that the operator
\begin{align}
  \widetilde\prec^\ell 
  := \sum_{z\in\widetilde\NN_\ell} \prec_z^\ell : \XX^L \to \widetilde \XX^\ell
\end{align}
and the matrix
\begin{align}
  \widetilde\PP^\ell = (\widetilde\DD^\ell)^{-1} (\widetilde{\bf I}^\ell)^T \AA^L
\end{align}
are related through $\edual{\widetilde\prec^\ell v}w =
\dual{\widetilde\PP^\ell \xx}\yy_{\AA^L}$.
In particular, the matrix $\PPAStilde^L$ reads
\begin{align}
  \PPAStilde^L  = \Big( \sum_{\ell=0}^L \widetilde{\bf I}^\ell
  (\widetilde\DD^\ell)^{-1} (\widetilde{\bf I}^\ell)^T \Big) \AA^L
 =: (\BB^L)^{-1} \AA^L,
\end{align}
with $\edual{\PAStilde^L v}w = \dual{\PPAStilde^L \xx}\yy_{\AA^L}$ for $w =
\sum_{k=0}^{N_L} \yy_k \eta_{z_k}^L$,
i.e.\ the additive Schwarz operator $\PAStilde^L$ generates the preconditioner from Theorem~\ref{thm:main}.

The following well-known result, see e.g.~\cite[Lemma~2]{griosw94}, collects some soft results on the additive Schwarz
operator $\PAStilde^L$ which hold in an arbitrary Hilbert space setting with 
\begin{align}
  \XX^L = \sum_{\ell=0}^L\sum_{z\in\widetilde\NN_\ell}\XX_z^\ell.
\end{align}

\begin{lemma}\label{lem:propPAS}
The operator $\PAStilde^L$ is linear and bounded as well as symmetric 
\begin{align}
  \edual{\PAStilde^Lv}{w} = \edual{v}{\PAStilde^Lw} 
 \quad\text{for all }v,w\in H^{1/2}(\Gamma)
\end{align}
and positive definite on $\XX^L$
\begin{align}
  \edual{\PAStilde^Lv}{v} > 0
 \quad\text{for all }v\in \XX^L\backslash\{0\}
\end{align}
with respect to the scalar product $\edual\cdot\cdot$.\qed
\end{lemma}

In our concrete setting, the additive Schwarz operator $\PAStilde^L$ 
satisfies the following spectral equivalence estimate which is proved in
Section~\ref{section:lower bound} (lower bound) resp.
Section~\ref{section:upper bound} (upper bound).
\begin{proposition}\label{prop:as}
  The operator $\PAStilde^L$ satisfies
\begin{align}\label{eq:prop:PAS}
  c \, \enorm{v}^2 \leq \edual{\PAStilde^L v}v 
 \le C\,\enorm{v}^2\quad\text{for all } v\in\XX^L.
\end{align}
The constants $c,C>0$ depend only on $\Gamma$ and the initial triangulation
$\TT_0$.
\end{proposition}

The relation between $\PAStilde^L$ and the symmetric matrix
$\PPAStilde^L$ 
yields the eigenvalue estimates from Theorem~\ref{thm:main}.

\begin{proof}[Proof of Theorem~\ref{thm:main}]
  Symmetry of $(\widetilde\BB^L)^{-1}$ follows from the definition~\eqref{eq:defBB}.
The other properties are obtained using the identity
\begin{align*}
  \edual{\PAStilde^L v}w = \dual{\PPAStilde^L \xx}\yy_{\AA^L} \quad\text{for all } v =
  \sum_{j=1}^{N_L} \xx_j \eta_{z_j}^L, w = \sum_{j=1}^{N_L} \yy_j \eta_{z_j}^L.
\end{align*}
An immediate consequence of this identity and Lemma~\ref{lem:propPAS} is
symmetry as well as positive definiteness of $\PPAStilde^L$ with
respect to $\dual\cdot\cdot_{\AA^L}$.
Proposition~\ref{prop:as} directly yields
\begin{align*}
  c \norm\xx{\AA^L}^2 \leq \dual{\PPAStilde^L \xx}\xx_{\AA^L} \leq C \norm\xx{\AA^L}^2
  \quad\text{for all }\xx\in \R^{N_L}.
\end{align*}
Thus, a bound for the minimal resp. maximal eigenvalue of
$\PPAStilde^L$ is given by
\begin{align*}
  \evmin(\PPAStilde^L)\geq c, \qquad \evmax(\PPAStilde^L) \leq C.
\end{align*}
In the next step, we prove positive definiteness of
$(\widetilde\BB^L)^{-1}$. Lemma~\ref{lem:propPAS} shows 
\begin{align*}
  0 < \edual{\PAStilde^L v}v = \dual{\PPAStilde^L
  \xx}\xx_{\AA^L} = \dual{(\widetilde\BB^L)^{-1}\AA^L \xx}{\AA^L \xx}_2.
\end{align*}
Since $\AA^L$ is regular, we obtain $\dual{(\widetilde\BB^L)^{-1}\yy}\yy_2 > 0$ for all
$\yy\in\R^{N_L}$. 
In particular the inverse $\widetilde\BB^L$ of
$(\widetilde\BB^L)^{-1}$ is well-defined, symmetric
and positive definite. The identity
\begin{align*}
  \dual{\PPAStilde^L \xx}\yy_{\widetilde\BB^L} = \dual{\AA^L \xx}\yy_2 \quad\text{for all }
  \xx,\yy\in\R^{N_L},
\end{align*}
and the symmetry of $\AA^L$ prove that $\PPAStilde^L$ is symmetric with
respect to $\dual\cdot\cdot_{\widetilde\BB^L}$.
Finally, we stress that the condition number $\cond_\CC(\AA)$ of a matrix $\AA$
is given by $\cond_\CC(\AA) = \evmax(\AA) / \evmin(\AA)$, if $\CC$ is positive
definite as well as symmetric and $\AA$ is symmetric with respect to
$\dual\cdot\cdot_\CC$. Therefore,
\begin{align*}
  \cond_{\widetilde\BB^L}(\PPAStilde^L) =
  \cond_{\AA^L}(\PPAStilde^L) =
  \frac{\evmax(\PPAStilde^L)}{\evmin(\PPAStilde^L)} \leq \frac{C}c,
\end{align*}
which also concludes the proof.
\end{proof}
The remainder of this section is now concerned with the proof of
Proposition~\ref{prop:as}.
This requires some preparations and auxiliary results
(Section~\ref{sec:proof:unif}--\ref{sec:proof:aux}), before we face the lower
bound (Section~\ref{section:lower bound}) and the upper bound
(Section~\ref{section:upper bound}) of~\eqref{eq:prop:PAS}.

%%%%%%%%%%%%%%%5
% uniform refinement & L2-Projection
% !TEX root = hypsingAS.tex

%--------------------------------------------------------------------------------
\subsection{Uniform mesh-refinement and $L^2$-orthogonal projection}
\label{sec:proof:unif}
%--------------------------------------------------------------------------------
Besides the sequence of locally refined triangulations $\TT_\ell$, we consider 
a second unrelated sequence $\widehat\TT_m$ of uniform triangulations:
Let $\widehat\TT_0 :=\TT_0$ and let $\widehat\TT_{m+1}$ be obtained from
$\widehat\TT_m$ by uniform refinement, i.e.\ all elements of $\widehat\TT_m$ are
bisected into son elements with half diameter. For $d=2$, this corresponds to
one bisection per element, while three bisections are used for $d=3$,
cf.\ Figure~\ref{fig:nvb:bisec}.
Let $\widehat\NN_m$ denote the set of all nodes of $\widehat\TT_m$.
Recall $\widehat h_0 := \max_{T\in\TT_0} h_0(T)$ and define the constant
\begin{align}\label{def:hat-hell}
  \widehat h_m := 2^{-m} \widehat h_0 \quad\text{for each }m \geq 1.
\end{align}
Note that $\widehat h_m$ is equivalent to the usual local mesh-size function
on $\widehat\TT_m$, i.e. $\widehat h_m \simeq \diam(T)$ for all
$T\in\widehat\TT_m$ and all $m\ge0$.
For later use, we recall the following result which is part 
of~\cite[Proof of Lemma~3.3]{wuchen06}.

\begin{lemma}\label{lem:unifEquivalence}
For all $z\in\NN_\ell$ holds $z\in\widehat\NN_{\level_\ell(z)}$ with
\begin{align}
  \c{eq_unif_1} \widehat h_{\level_\ell(z)} \leq h_\ell(z) \leq 
  \c{eq_unif_2} \widehat h_{\level_\ell(z)},
\end{align}
where the constants $\setc{eq_unif_1},\setc{eq_unif_2}>0$ depend only on the $\gamma$-shape regularity and the initial
triangulation $\TT_0$.\qed
\end{lemma}

Let $\widehat\XX^m := \SS^1(\widehat\TT_m)$ and denote by
$\widehat\Pi_m : L^2(\Gamma)\to\widehat\XX^m$
the $L^2$-orthogonal projection onto $\widehat\XX^m$. 
%We set $\widehat\Pi_\ell := \Pi_0$ if $\ell<0$.
%\question{WOZU IST DIE LETZTE DEFINITION N\"OTIG? HIER NICHT?}

\begin{lemma}\label{lem:h12normest}
  For all $v\in H^{1/2}(\Gamma)$ holds
  \begin{align}\label{eq:h12normest}
    \sum_{m=0}^\infty \widehat h_m^{-1} 
    \norm{v-\widehat\Pi_m v}{L^2(\Gamma)}^2
    \le \c{norm} \, \norm{v}{H^{1/2}(\Gamma)}^2.
  \end{align}
  The constant $\setc{norm}>0$ depends only on $\Gamma$ and the
  initial triangulation $\TT_0$.
\end{lemma}

\begin{proof}
  We note that $\widehat\XX_k\subseteq\widehat\XX_{k+1}$ and
  $\lim\limits_{k\to\infty} \norm{v-\widehat\Pi_k v}{L^2(\Gamma)} = 0$ for
  all $v\in L^2(\Gamma)$. Thus,
  \begin{align}\label{eq:h12normest1}
    \norm{v-\widehat\Pi_m v}{L^2(\Gamma)}^2 = \sum_{k=m+1}^\infty
    \norm{(\widehat\Pi_k-\widehat\Pi_{k-1})v}{L^2(\Gamma)}^2.
  \end{align}
  Plugging~\eqref{eq:h12normest1} into the left-hand side
  of~\eqref{eq:h12normest} and changing the order of summation, we see
  \begin{align}\label{eq:h12normest2}
    \begin{split}
    \sum_{m=0}^\infty \widehat h_m^{-1}\norm{v-\widehat\Pi_m v}{L^2(\Gamma)}^2
    &= \sum_{m=0}^\infty \widehat h_m^{-1} \sum_{k=m+1}^\infty
    \norm{(\widehat\Pi_k-\widehat\Pi_{k-1})v}{L^2(\Gamma)}^2 \\
    &= \sum_{k=1}^\infty \Big(\sum_{m=0}^{k-1} \widehat h_m^{-1}\Big)
    \norm{(\widehat\Pi_k-\widehat\Pi_{k-1})v}{L^2(\Gamma)}^2.
    \end{split}
  \end{align}
  With the definition~\eqref{def:hat-hell} of $\widehat h_m$ and the geometric series we infer
  \begin{align}\label{eq:h12normest3}
    \sum_{m=0}^{k-1} \widehat h_m^{-1} = \widehat h_0^{-1} \sum_{m=0}^{k-1} 2^m 
    < \widehat h_0^{-1} 2^k = \widehat h_k^{-1}.
  \end{align}
  \cite[Theorem~5]{amcl03} states that for
  $s\in[0,1]$ and $v\in H^s(\Gamma)$ it holds
  \begin{align}\label{eq:h12normest4}
    \norm{v}{H^s(\Gamma)}^2 \simeq \norm{\widehat\Pi_0 v}{H^s(\Gamma)}^2 +
    \sum_{k=1}^\infty \widehat h_k^{-2s}
    \norm{(\widehat\Pi_k-\widehat\Pi_{k-1})v}{L^2(\Gamma)}^2.
  \end{align}
  The hidden constants in~\eqref{eq:h12normest4} depend only on $\Gamma$, the
  initial triangulation $\TT_0$, and on $s$.
  Using equations~\eqref{eq:h12normest2}--\eqref{eq:h12normest3} and norm
  equivalence~\eqref{eq:h12normest4} for $s=1/2$, we conclude the proof of~\eqref{eq:h12normest}.
\end{proof}

%%%%%%%%%%%%%%%5
% Scott-Zhang
% !TEX root = hypsingAS.tex

%--------------------------------------------------------------------------------
\subsection{Scott-Zhang projection}\label{sec:sz}
%--------------------------------------------------------------------------------
We require a variant of the Scott-Zhang quasi-interpolation 
operator~\cite{sz}, see also~\cite{hypsing3d} for higher-order polynomials
$\SS^p(\TT_\ell)$ for $p\ge1$
and $H^s(\Gamma)$ resp.\ $\widetilde H^s(\Gamma)$ with $\Gamma\subseteq\partial\Omega$:
For $z\in\NN_\ell$, let
$T_z^\ell\in\TT_\ell$ be an element with $z\in T_z^\ell$. Let $\psi_z^\ell$
denote the $L^2$-dual basis function with 
\begin{align}\label{eq:szdual}
  \int_{T_z^\ell} \psi_z^\ell(x) \eta_{z'}^\ell(x) \,ds_x =
  \delta_{zz'} \quad\text{for all } z'\in\NN_\ell.
\end{align}
Then, the operator $J_\ell : L^2(\Gamma) \to \SS^1(\TT_\ell)$ defined by
\begin{align}\label{eq:szdef}
  J_\ell v = \sum_{z\in\NN_\ell}\eta_z^\ell \int_{T_z^\ell} \psi_z^\ell(x) v(x) \,ds_x
  \quad\text{for all }v\in L^2(\Gamma),
\end{align}
is an $H^s$-stable projection onto $\SS^1(\TT_\ell)$, i.e.
\begin{align}
 J_\ell v_\ell = v_\ell
 \quad\text{and}\quad
 \norm{J_\ell v}{H^s(\Gamma)}
 \le \c{stable}\,\norm{v}{H^s(\Gamma)}
 \quad\text{for all }
 v_\ell\in\SS^1(\TT_\ell)\text{ and }v\in H^s(\Gamma).
\end{align}
Arguing as in~\cite{sz}, $J_\ell$ satisfies for all $v\in H^1(\Gamma)$
\begin{align}
 \norm{\nabla J_\ell v}{L^2(T)}
 \le\c{sz}\,\norm{\nabla v}{L^2(\omega_\ell(T))}
 \quad\text{and}\quad
 \norm{v-J_\ell v}{L^2(T)}
 \le\c{sz}\,h_\ell(T)\,\norm{\nabla v}{L^2(\omega_\ell(T))}.
\end{align}
The constant $\setc{sz}>0$ depends only on $\gamma$-shape regularity
of $\TT_\ell$, while $\setc{stable}>0$ additionally depends on $\Gamma$.
Moreover, if $v$ is linear on $T_z^\ell$ it holds
\begin{align}\label{eq:szlin1}
  J_\ell v(z) = v(z).
\end{align}
Note that the choice of $T_z^\ell$ is arbitrary, but for
$z\in\NN_{\ell}\backslash\widetilde\NN_\ell \subseteq \NN_{\ell-1}$ we require
that $T_z^{\ell-1} = T_z^\ell \in \TT_\ell\cap \TT_{\ell-1}$. 
For $z\in\NN_\ell\backslash\widetilde\NN_\ell$, it thus follows $\eta_z^\ell =
\eta_z^{\ell-1}$ as well as $T_z^\ell = T_z^{\ell-1}$ and consequently $\psi_z^\ell =
\psi_z^{\ell-1}$.
This yields the following important
property that allows us to construct a stable subspace decomposition:
\begin{align}\label{eq:szprop1}
  (J_\ell-J_{\ell-1})v(z) = 0 \quad\text{for all }z\in\NN_\ell\backslash\widetilde\NN_\ell.
\end{align}
In particular, we have 
\begin{align}\label{eq:szprop1b}
 (J_\ell-J_{\ell-1})v\in 
 \linhull\{ \eta_z^\ell \,:\: z\in\widetilde\NN_\ell\} = \widetilde\XX^\ell.
\end{align}
Finally and with the second-order node patch 
from~\eqref{eq:patch},
we have the following pointwise estimate for the Scott-Zhang projection.

\begin{lemma}\label{lem:szest}
  For all $v\in L^2(\Gamma)$, it holds
  \begin{align}\label{eq:szest}
    |(J_\ell-J_{\ell-1})v(z)| \le \c{sz:point}\, h_\ell(z)^{-(d-1)/2}
    \norm{v}{L^2(\omega_{\ell-1}^2(z))} \quad\text{for all } z\in\widetilde\NN_\ell.
  \end{align}
  The constant $\setc{sz:point}>0$ depends only on $\gamma$-shape
  regularity of $\TT_\ell$.
\end{lemma}

\begin{proof}
According to~\cite[Lemma~3.1]{sz}, it holds $\norm{\psi_z^\ell}{L^\infty(T_z^\ell)}
\lesssim |T_z^\ell|^{-1}$. For any node $z\in\widetilde\NN_\ell$, we have
$T_z^\ell\subseteq \omega_\ell(z)\subseteq \omega_{\ell-1}^2(z)$ and thus
\begin{align}\label{eq:szest1}
  \begin{split}
  |J_\ell v(z)| &\leq \int_{T_z^\ell} |\psi_z^\ell(x) v(x)| \,ds_x \leq
  \norm{\psi_z^\ell}{L^\infty(T_z^\ell)} |T_z^\ell|^{1/2}
  \norm{v}{L^2(T_z^\ell)} \\
  &\lesssim |T_z^\ell|^{-1/2} \norm{v}{L^2(\omega_{\ell-1}^2(z))} \lesssim
  h_\ell(z)^{-(d-1)/2} \norm{v}{L^2(\omega_{\ell-1}^2(z))}.
  \end{split}
\end{align}
For $z\in\widetilde\NN_\ell \backslash \NN_{\ell-1}$, there exist two nodes
$z_1,z_2\in\NN_{\ell-1}$ such that
\begin{align*}%\label{eq:szest2}
  J_{\ell-1}v (z) = \eta_{z_1}^{\ell-1}(z) \int_{T_{z_1}^{\ell-1}}
  \psi_{z_1}^{\ell-1}(x)v(x) \,ds_x
  + \eta_{z_2}^{\ell-1}(z) \int_{T_{z_2}^{\ell-1}}
  \psi_{z_2}^{\ell-1}(x)v(x) \,ds_x.
\end{align*}
For $z\in\widetilde\NN_\ell\cap\NN_{\ell-1}$, 
this equality is understood with
$z_1=z$ and
$\eta_{z_2}^{\ell-1}=0$. In either case, we note that $|T_{z_i}^{\ell-1}|
\simeq h_\ell^{d-1}(z)$ as well as 
$T_{z_i}^{\ell-1}\subseteq \omega_{\ell-1}(z_i)\subseteq\omega_{\ell-1}^2(z)$.
Analogously to~\eqref{eq:szest1}, we derive
\begin{align}\label{eq:szest3}
  |J_{\ell-1}v(z)| \lesssim |T_{z_1}^{\ell-1}|^{-1/2}
  \norm{v}{L^2(T_{z_1}^{\ell-1})} + |T_{z_2}^{\ell-1}|^{-1/2}
  \norm{v}{L^2(T_{z_2}^{\ell-1})} \lesssim h_\ell(z)^{-(d-1)/2}
  \norm{v}{L^2(\omega_{\ell-1}^2(z))}.
\end{align}
Combining the triangle inequality $|(J_\ell-J_{\ell-1})v(z)| \leq |J_\ell v(z)|
+ |J_{\ell-1}v(z)|$ with~\eqref{eq:szest1}--\eqref{eq:szest3},
we prove~\eqref{eq:szest}.
\end{proof}

%\input{04_proof_invest}

% !TEX root = hypsingAS.tex

\begin{figure}[t]
 \centering
 \includegraphics[width=35mm]{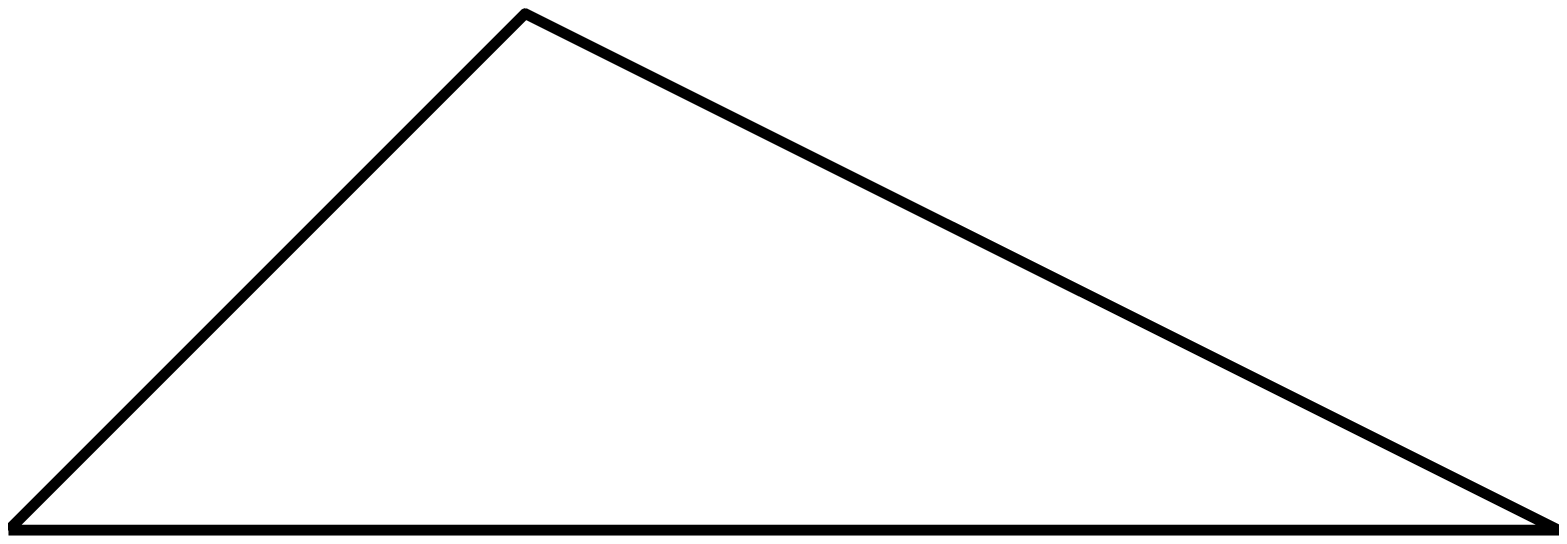}\quad
 \includegraphics[width=35mm]{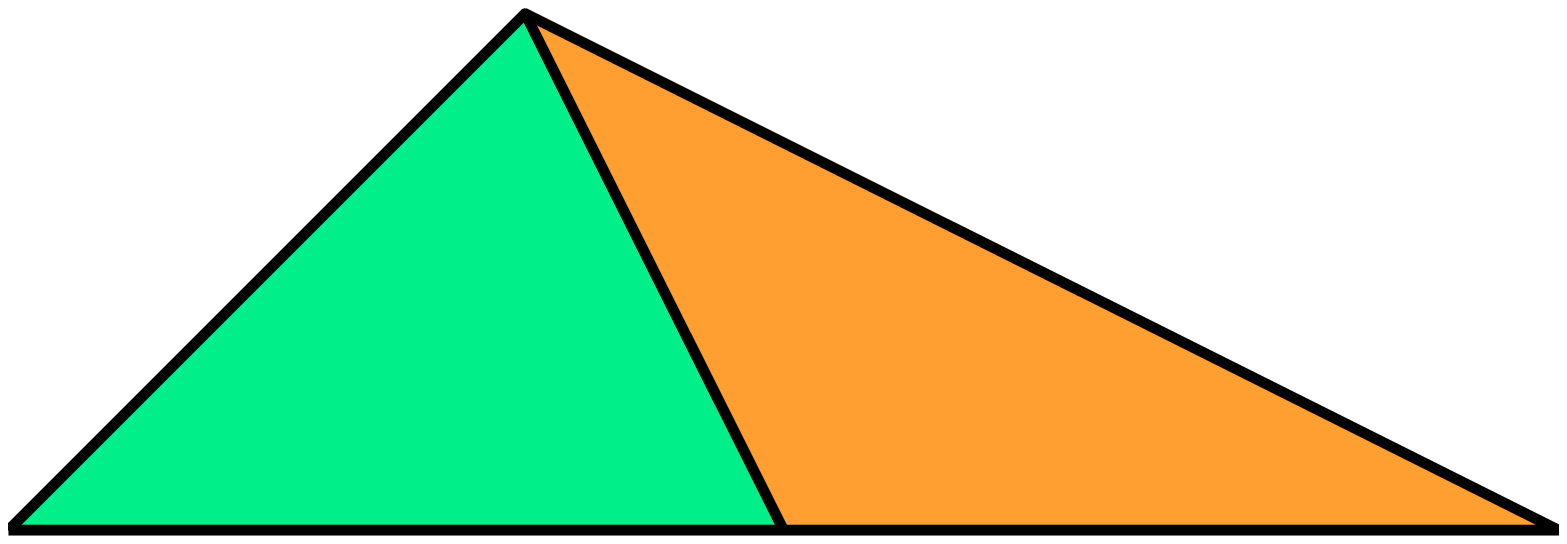}\quad
 \includegraphics[width=35mm]{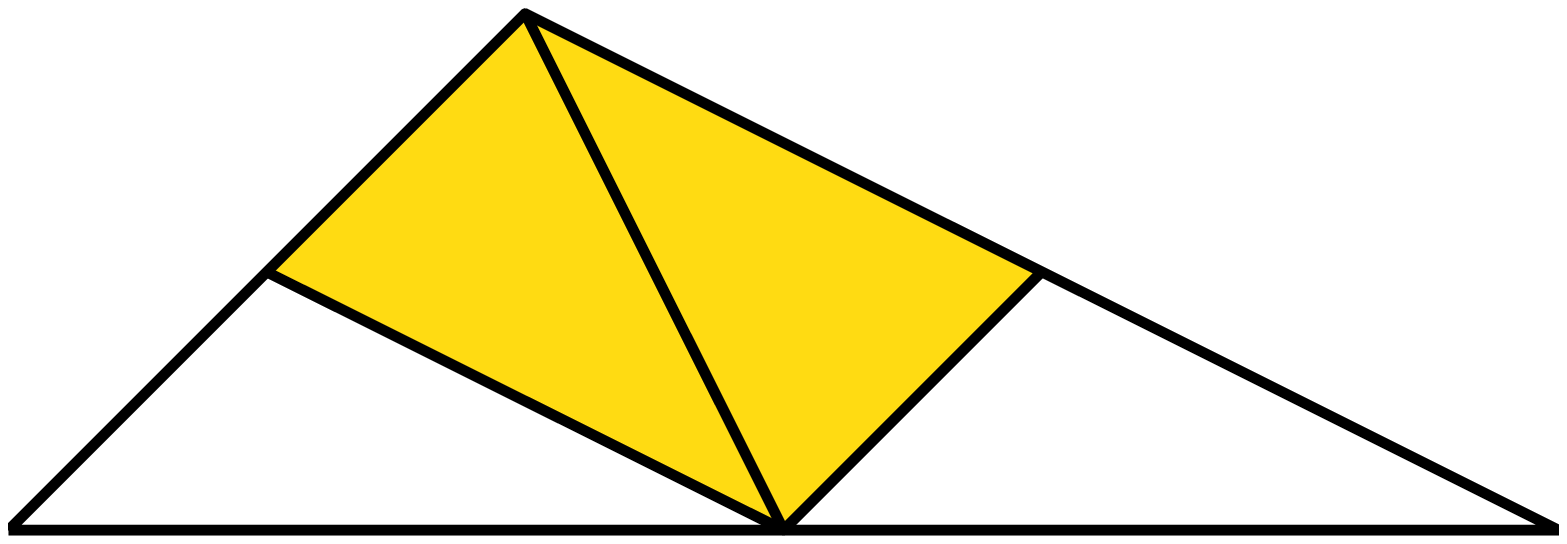}\quad
 \includegraphics[width=35mm]{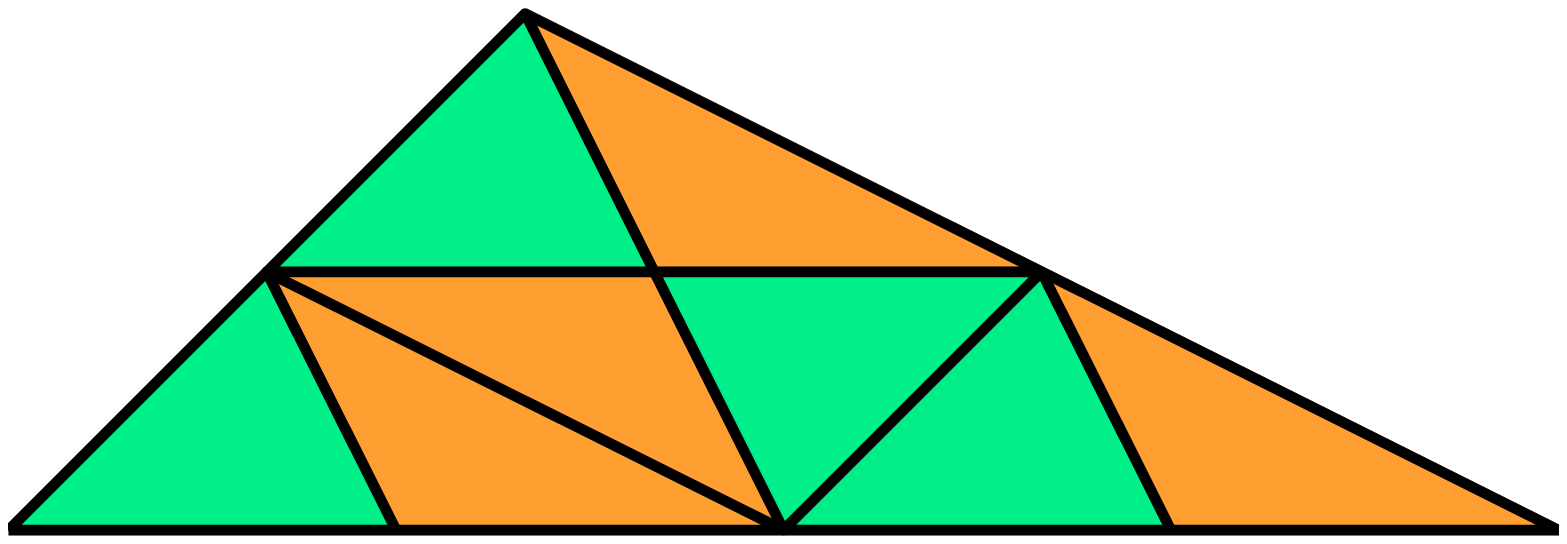}
 \caption{Refinement by newest vertex bisection leads only to
 finitely many similarity classes of triangles.
 To see this, we start from a macro element (left), where the bottom edge is the
 reference edge. Using iterated newest vertex bisection, one observes
 that only four similarity classes of triangles occur, which
 are indicated by the coloring. After three levels of bisection
 (right), no additional similarity class appears.}
 \label{fig:nvb:similarity}
\end{figure}

\begin{figure}[th]
  \begin{center}
    \psfrag{z}{\tiny $z$}
    \psfrag{zp}{\tiny $z'$}
    \psfrag{T0}{\tiny $\TT_0$}
  \includegraphics[width=0.32\textwidth]{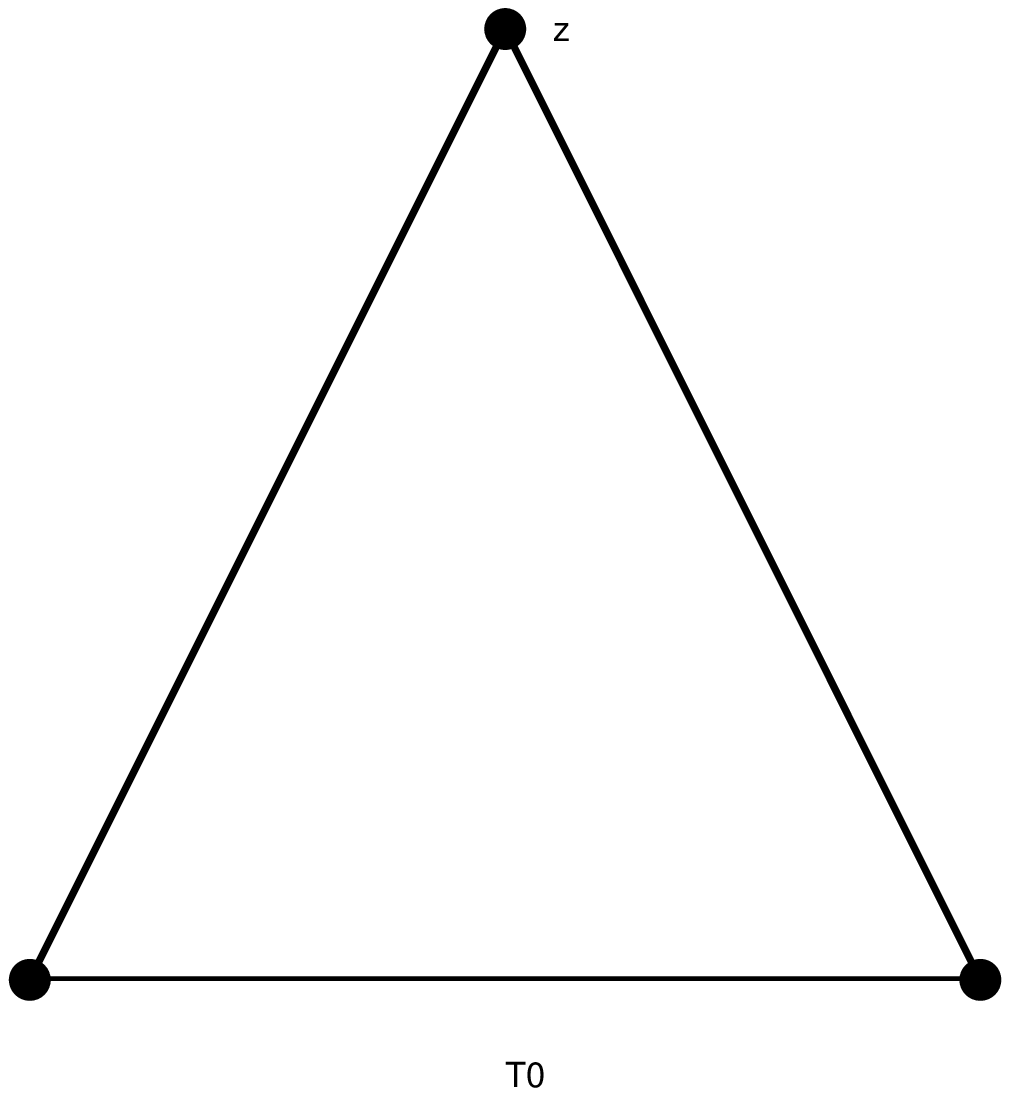}
    \psfrag{T1}{\tiny $\TT_1$}
  \includegraphics[width=0.32\textwidth]{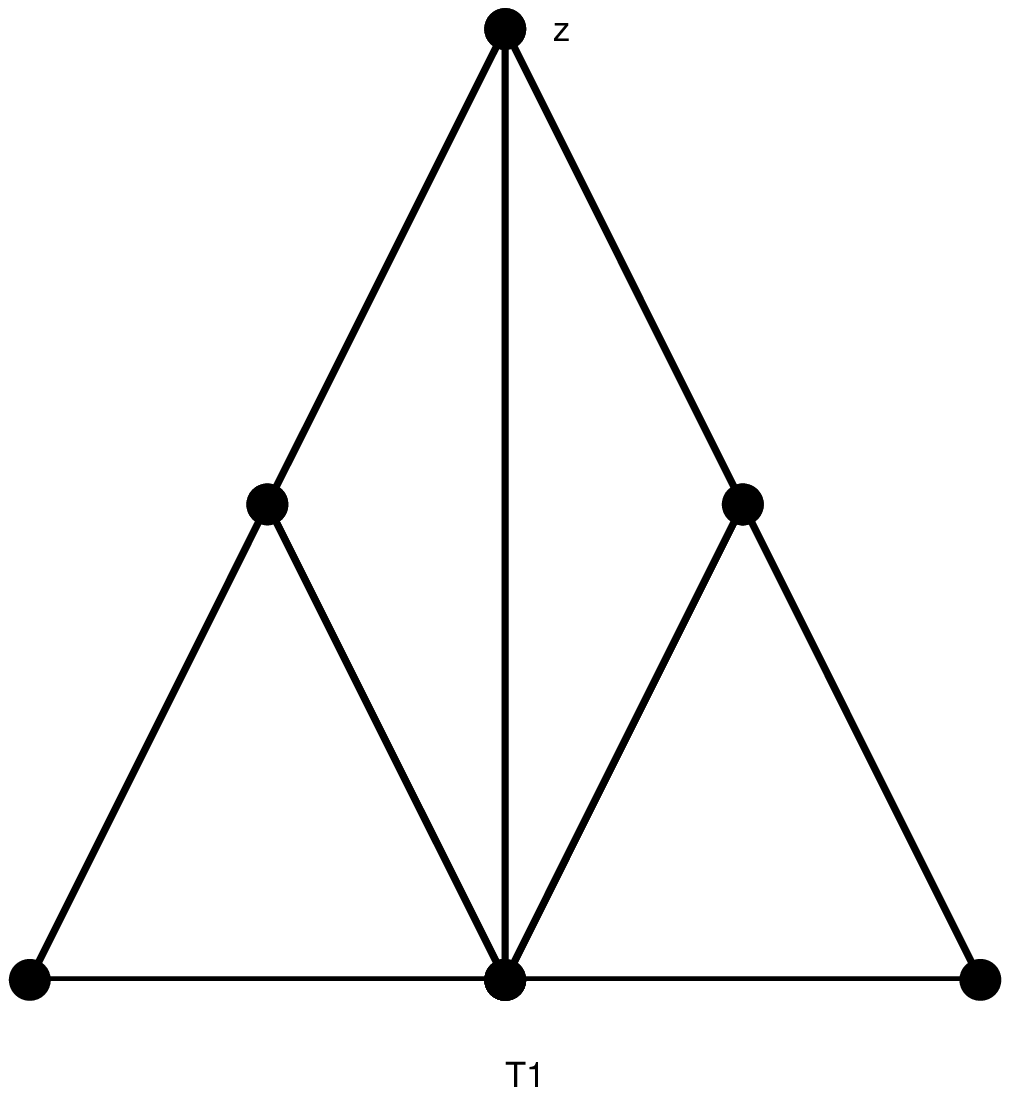}
    \psfrag{T2}{\tiny $\TT_2$}
  \includegraphics[width=0.32\textwidth]{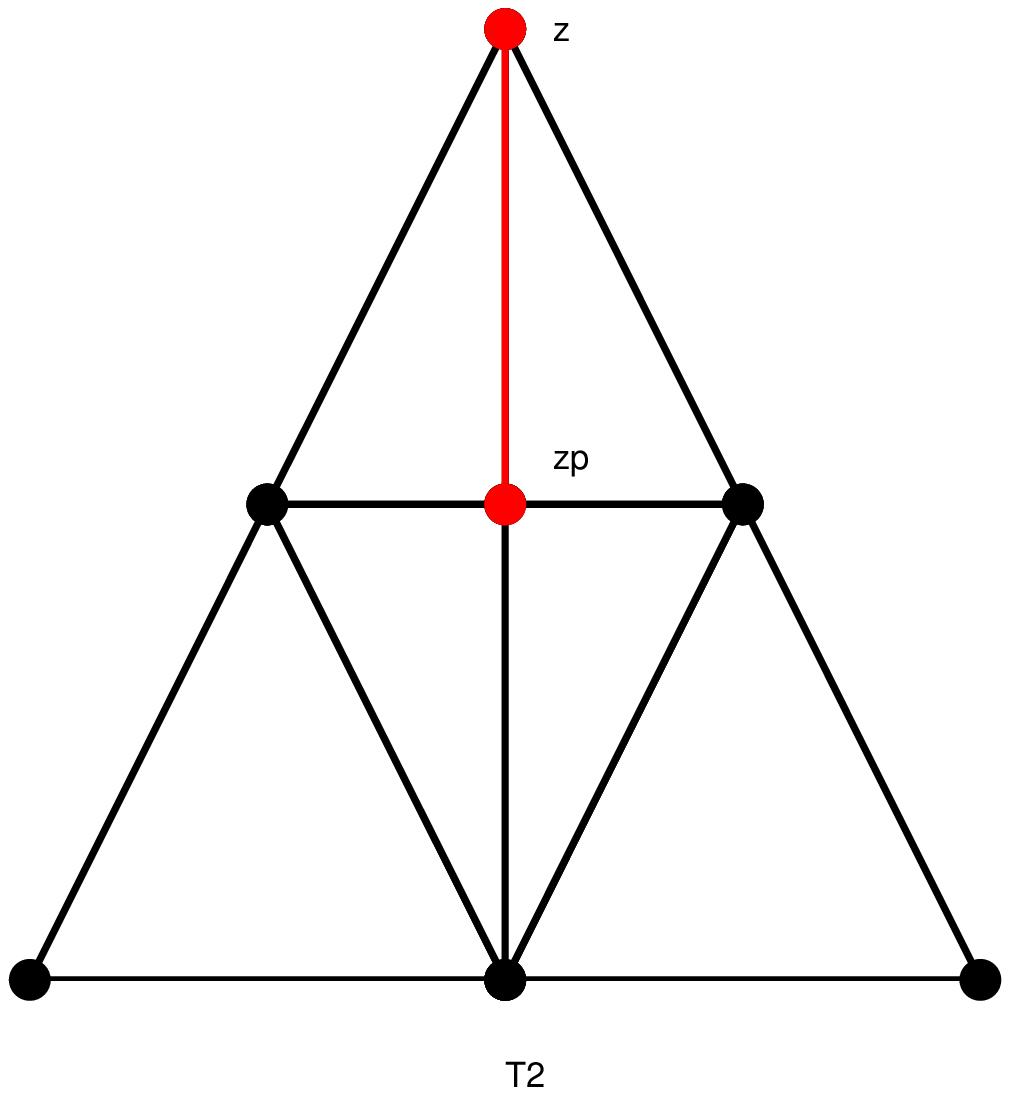}
  \caption{We consider the sequence of triangulations $\TT_0,\TT_1,\TT_2$. The
  initial triangulation $\TT_0$ consists of one triangle. The edges incident to
  $z\in\NN_0$ have length $\sqrt{5}/2 e>0$, whereas the edge opposite to $z$
  has length $e>0$. 
  The triangulation $\TT_1$ is obtained from $\TT_0$ by uniform refinement and
  $\TT_2$ is obtained from $\TT_1 = \widehat \TT_1$ by bisecting the longest edge incident to $z$.
  The shortest edge incident to $z$ in the triangulation $\TT_2$, i.e. $h_2(z)$, is
  spanned by the nodes $z,z'$ (red).
  A simple calculation shows $\level_2(z) = 1$. Obviously, $z'\notin \widehat
  \NN_1 = \widehat \NN_{\level_2(z)}$ and the corresponding
  hat-function $\eta_z^2$ is not an element of $\widehat\XX^1 =
  \widehat\XX^{\level_2(z)}$.
  This shows that the result of Lemma~\ref{lem:hatunif} cannot be improved.
  }
  \label{fig:counterEx}
\end{center}
\end{figure}

%--------------------------------------------------------------------------------
\subsection{Further auxiliary results}\label{sec:proof:aux}
%--------------------------------------------------------------------------------
The proof of Proposition~\ref{prop:as} requires some additional definitions
and technical results.
For a given node $z\in\widetilde\NN_\ell$, it may hold $z\in\widetilde\NN_{\ell+m}$ 
even with the same level $\level_\ell(z) = \level_{\ell+m}(z)$.
We count how often a node $z\in\NN_L$ with a fixed level $k\in\N_0$ shows up in
the sets $\widetilde\NN_\ell$. For $z\in\NN_L$ and $k\in\N_0$, we therefore define
\begin{align}\label{def:KK}
  \widetilde\KK_k(z) := \set{ \ell\in \{0,1,\dots,L\}}{z\in\widetilde\NN_\ell
  \quad\text{and}\quad \level_\ell(z) = k}.
\end{align}
The following lemma from~\cite[Lemma~3.1]{wuchen06} proves that the 
cardinality of these set $\widetilde\KK_k(z)$ is uniformly bounded.

\begin{lemma}\label{lem:Kbound}
For all $z\in\NN_L$ and $k\in\N_0$, it holds $\# \widetilde\KK_k(z) \leq \c{Cbound}$, and the
constant $\setc{Cbound}>0$ depends only on the initial triangulation $\TT_0$, but is independent
of $L$, $k$, and $z$.\qed
\end{lemma}

For each $T\in\TT_\ell$, there exists a unique coarse-mesh ancestor $T_0\in\TT_0$ 
with $T\subseteq T_0$. For $d=2$, there exists some $k\in\N_0$ such that 
$h_\ell(T) = 2^{-k} h_0(T_0)$, i.e., $T$ is created by $k$ bisections of $T_0$ and $k$
thus is the level of $T$.
For $d=3$, the successors of $T_0$ belong to at most four similarity classes, see
Figure~\ref{fig:nvb:similarity}. Let $\widetilde T_0\subset T_0$ be the unique successor 
of $T_0$ with $T\subseteq \widetilde T_0$ which is similar
to $T$ and has the maximal diameter with respect to its similarity class. 
Then, there exists $k\in\N_0$ such that
$h_\ell(T) = 2^{-k} \diam(\widetilde T_0)$. For both cases ($d=2,3$), we define
$r(T) = k$.

For each $z\in\NN_\ell$, we further define the quantity
$R_\ell(z)$ associated to the patch $\omega_{\ell-1}^2(z)$ as 
\begin{align}
  R_\ell(z) := \min\set{r(T)}{ T\in \TT_{\ell-1}, T\subseteq
  \omega_{\ell-1}^2(z)}.
\end{align}
The proof of the following lemma can be found in~\cite[Proof of
Lemma~3.3]{wuchen06} for $d=3$. For $d=2$, the proof follows the same lines.
Details are left to the reader.

\begin{lemma}\label{lem:patchlevel}
{\rm(i)} For all $z\in\NN_\ell$ holds $\level_\ell(z) \leq R_\ell(z) + \c{n}$, and the constant
$\setc{n}>0$ depends only on the initial triangulation $\TT_0$. \\
{\rm(ii)} For all $z\in\NN_\ell$ and $T\in\TT_{\ell-1}$ with $T \subseteq
\omega_{\ell-1}^2(z)$, there exists an element $\widehat T\in\widehat \TT_{R_\ell(z)}$ such
that $T\subseteq \widehat T$.\qed
\end{lemma}

Analogously to the patch $\omega_\ell^k(z)$ from~\eqref{eq:patch}, we define 
the patch $\widehat\omega_m^k(z)$ corresponding to the uniformly refined triangulation
$\widehat\TT_m$.

\begin{lemma}\label{lem:patchinclusion}
There exists $n\in\N_0$, which depends only on the initial triangulation
$\TT_0$, such that for all $z\in\NN_\ell$ holds
$\omega_\ell(z) \subseteq \omega_{\ell-1}^2(z) \subseteq
\widehat\omega_{\level_\ell(z)}^n (z)$.
\end{lemma}
\begin{proof}
Obviously, $\omega_\ell(z) \subseteq \omega_{\ell-1}(z) \subseteq
\omega_{\ell-1}^2(z)$ by definition of these patches.
It remains to prove the second inclusion $\omega_{\ell-1}^2(z) \subseteq
\widehat\omega_{\level_\ell(z)}^n(z)$:
Lemma~\ref{lem:patchlevel} states that for $T\in \TT_{\ell-1}$ with $T\subseteq \omega_{\ell-1}^2(z)$, there
exists $\widehat T \in \widehat\TT_{R_\ell(z)}$ with $T\subseteq \widehat T \subseteq 
\widehat \omega_{R_\ell(z)}^2(z)$.
Each element $\widehat T \in \TT_{R_\ell(z)}$ with $\widehat T\subseteq \widehat
\omega_{R_\ell(z)}^2(z)$ is bisected into $2(d-1)\c{n}$ elements $\widehat
T'_j \in \widehat \TT_{R_\ell(z)+ \c{n}}$ such that
\begin{align}
  \widehat T = \bigcup_{j=1}^{2(d-1)\c{n}} \widehat T'_j.
\end{align}
In particular, there exists a constant $n\in \N$ with $n\leq 4(d-1)\c{n}$
such that $\widehat T \subseteq \widehat\omega_{R_\ell(z)+\c{n}}^n(z)$.
Lemma~\ref{lem:patchlevel} states that $\level_\ell(z) \leq R_\ell(z) + \c{n}$.
Hence, $\widehat\omega_{R_\ell(z)+\c{n}}^n(z) \subseteq
\widehat\omega_{\level_\ell(z)}^n(z)$ by the definition of the patches.
\end{proof}

\begin{lemma}\label{lem:hatunif}
For all nodes $z\in\NN_\ell$ holds $\eta_z^\ell \in \widehat
\XX^{m+1}$ with $m = \level_\ell(z)$.
\end{lemma}
\begin{proof}[Proof for $d=2$]
We prove that all nodes $z'\in \NN_\ell\cap \omega_\ell(z) \backslash \{z\}$
satisfy $z' \in \widehat\NN_m$. This implies $\eta_z^\ell \in \widehat
\XX^m$. 
Let $T_{zz'} \in \TT_\ell$ denote the element which is spanned
by the nodes $z,z'$. According to bisection, there exists a coarse mesh-element $T_0 \in\TT_0$ and $\widehat n \in \N_0$ 
such that $T_{zz'} \in \widehat\XX^{\widehat n}$ with $T_{zz'} \subseteq T_0$,
and $|T_{zz'}| = |T_0| \, 2^{-\widehat n}$.
The definition of $\level_\ell(z)$ and $h_\ell(z)$ yields
\begin{align*}
  2^{-m-1} \widehat h_0 < h_\ell(z) \leq |T_{zz'}| = |T_0| \, 2^{-\widehat n} \leq
  \widehat h_0 2^{-\widehat n}
\end{align*}
and consequently $m+1> \widehat n$. Thus $m\geq \widehat n$, and we conclude
$z'\in\widehat \NN_m \subseteq \widehat\NN_{m+1}$.
\end{proof}
\begin{proof}[Proof for $d=3$]
Let $E_{zz'}$ denote the
edge between $z\in\NN_\ell$ and $z'\in\NN_\ell \cap \omega_\ell(z) \backslash
\{z\}$. There exists $\widehat n \in \N_0$ such that $E_{zz'}$ is an edge of $\widehat
\XX^{\widehat n}$. Furthermore, there exist edges $E_0,E_1$ with $E_0 \subseteq
T_0 \in \TT_0$, $E_1\subseteq T_1 \in \TT_1$ and $E_1$ is a median of the
macro element $T_0$ such that one of the following
cases holds:
\begin{itemize}
  \item[(i)] $|E_{zz'}| = |E_0| \, 2^{-\widehat n}$ or
  \item[(ii)] $|E_{zz'}| = |E_1| \, 2^{-\widehat n+1}$.
\end{itemize}
As in the case $d=2$, we obtain $z' \in \widehat \NN_m$ for (i) and $z'\in\widehat
\NN_{m+1}$ for (ii).
Since $\widehat\NN_m \subseteq \widehat\NN_{m+1}$, this proves $z'\in \widehat\NN_{m+1}$ for all $z'\in\NN_\ell \cap
\omega_\ell(z)$. Hence, $\eta_z^\ell \in \widehat \XX^{m+1}$.
\end{proof}
In general, the result of the previous Lemma cannot be improved in the sense that there
exists a number $k\in\N_0$ with $k<m+1$ such that $\eta_z^\ell \in \widehat\XX^k$.
See also Figure~\ref{fig:counterEx} for an example.

%%%%%%%%%%%%%%%%%%%5
% Lower bound
%
% !TEX root = hypsingAS.tex

%--------------------------------------------------------------------------------
\subsection{Proof of Proposition~\ref{prop:as}, lower bound}
\label{section:lower bound}
%--------------------------------------------------------------------------------
In this section, we prove the lower bound in the spectral equivalence
estimate~\eqref{eq:prop:PAS}. The proof relies on the 
following result which is also known as Lions' lemma~\cite{lions88,wid89}.

\begin{lemma}\label{lem:stabledec}
Suppose that each $v\in\XX^L$ admits a representation $v = \sum_{\ell=0}^L
\sum_{z\in\widetilde\NN_\ell} v_z^\ell$ with $v_z^\ell
\in\XX_z^\ell$ and
\begin{align}
  \sum_{\ell=0}^L \sum_{z\in\widetilde\NN_\ell} \enorm{v_z^\ell}^2 \leq
  c^{-1} \enorm{v}^2.
\end{align}
Then, $\PAStilde^L$ is elliptic
\begin{align}
  c\, \enorm{v}^2 \leq \edual{\PAStilde^L v}v
 \quad\text{for all }v\in\XX^L,
\end{align}
and the minimal eigenvalue of the additive Schwarz
operator satisfies $\evmin(\PAStilde^L) \geq c > 0$.
\qed
\end{lemma}

\begin{proof}[Proof of lower bound in~\eqref{eq:prop:PAS}]
Let $v\in\XX^L$ and set $J_{-1}:=0$. With the property~\eqref{eq:szprop1b} of the 
Scott-Zhang projection $J_\ell$, we define
\begin{align}
 \widetilde v^\ell := (J_\ell-J_{\ell-1})v \in \widetilde\XX^\ell
 \quad\text{for all } 0\leq\ell\leq L.
\end{align}
The projection property of $J_L$ and the telescoping series prove
\begin{align}
 v = J_Lv = (J_L-J_{-1})v = \sum_{\ell=0}^L \widetilde v^\ell.
\end{align}
We further decompose $v$ into
\begin{align}\label{eq:lb1}
  v = \sum_{\ell=0}^L \sum_{z\in\widetilde\NN_\ell} \widetilde v^\ell(z)\eta_z^\ell =:
  \sum_{\ell=0}^L\sum_{z\in\widetilde\NN_\ell} v_z^\ell 
  \quad\text{with } v_z^\ell \in\XX_z^\ell.
\end{align}
Let $z\in\NN_\ell$. According to the properties~\eqref{eq:hatfunprop} of the hat-functions, standard 
interpolation techniques yield
\begin{align}\label{eq:lb2}
 \enorm{\eta_z^\ell}^2
 \simeq \norm{\eta_z^\ell}{H^{1/2}(\Gamma)}^2
 \le \norm{\eta_z^\ell}{L^2(\Gamma)}\norm{\eta_z^\ell}{H^1(\Gamma)}
 \lesssim |\omega_\ell(z)|\,h_\ell(z)^{-1} 
 \simeq h_\ell(z)^{d-2}.
\end{align}
This implies
\begin{align*}
 \enorm{v_z^\ell}^2
 \lesssim h_\ell(z)^{d-2}|(J_\ell-J_{\ell-1})v(z)|^2.
\end{align*}
We set $\widehat\Pi_m := \widehat\Pi_0$ for $m<0$. 
Lemma~\ref{lem:patchlevel} yields
that $\widehat\Pi_{\level_\ell(z)-\c{n}} v \in \widehat\XX_{R_\ell(z)}$
and that $(\widehat\Pi_{\level_\ell(z)-\c{n}} v)|_T$ is linear for all
$T\in\TT_{\ell-1}$ with $T\subseteq \omega_{\ell-1}^2(z)$.
The Scott-Zhang projection preserves linearity on all elements
$T\in\TT_{\ell-1}$ resp. $T\in\TT_\ell$ with $z\in T\subseteq
\omega_{\ell-1}^2(z)$. This and Lemma~\ref{lem:szest} yield
\begin{align*}
 |(J_\ell-J_{\ell-1})v(z)|^2 &=
 |(J_\ell-J_{\ell-1})(v-\widehat\Pi_{\level_\ell(z)-\c{n}}v)(z)|^2
 \\&\lesssim h_\ell(z)^{-(d-1)}\,
 \norm{v-\widehat\Pi_{\level_\ell(z)-\c{n}}v}{L^2(\omega_{\ell-1}^2(z))}^2.
\end{align*}
Combining the last two estimates, we obtain
\begin{align*}%\label{eq:lb6}
  \enorm{v^\ell_z}^2
 \lesssim h_\ell(z)^{-1}\,
 \norm{v-\widehat\Pi_{\level_\ell(z)-\c{n}}v}{L^2(\omega_{\ell-1}^2(z))}^2
\end{align*}
Using the equivalence $h_\ell(z) \simeq \widehat h_{\level_\ell(z)}$ from 
Lemma~\ref{lem:unifEquivalence}, we get
\begin{align*}
  \sum_{\ell=0}^L \sum_{z\in\widetilde\NN_\ell} \enorm{v_z^\ell}^2
  &\lesssim \sum_{\ell=0}^L \sum_{z\in\widetilde\NN_\ell}
  h_\ell(z)^{-1} \norm{v-\widehat\Pi_{\level_\ell(z)-\c{n}}v}{L^2(\omega_{\ell-1}^2(z))}^2
  \\
  &\simeq \sum_{\ell=0}^L \sum_{z\in\widetilde\NN_\ell} \widehat h_{\level_\ell(z)}^{-1}
  \norm{v-\widehat\Pi_{\level_\ell(z)-\c{n}}v}{L^2(\omega_{\ell-1}^2(z))}^2
  \\
  &= \sum_{m=0}^\infty \sum_{\ell=0}^L \sum_{\substack{z\in\widetilde\NN_\ell\\ \level_\ell(z) = m}} 
  \widehat h_m^{-1} \norm{v-\widehat\Pi_{m-\c{n}}v}{L^2(\omega_{\ell-1}^2(z))}^2.
\end{align*}
With Lemma~\ref{lem:patchinclusion} and the definition~\eqref{def:KK} of
$\widetilde\KK_m(z)$, we see
\begin{align*}
  \sum_{m=0}^\infty \sum_{\ell=0}^L \!\!\sum_{\substack{z\in\widetilde\NN_\ell\\ \level_\ell(z) = m}}\!\!
  \widehat h_m^{-1} \norm{v-\widehat\Pi_{m-\c{n}}v}{L^2(\omega_{\ell-1}^2(z))}^2
  &\lesssim \sum_{m=0}^\infty \sum_{\ell=0}^L \!\!\sum_{\substack{z\in\widetilde\NN_\ell\\ \level_\ell(z) = m}} \!\!
  \widehat h_m^{-1} \norm{v-\widehat\Pi_{m-\c{n}}v}{L^2(\widehat\omega_m^n(z))}^2
  \\ 
  &= \sum_{m=0}^\infty \sum_{z\in\NN_L} \sum_{\ell\in \widetilde\KK_m(z)} 
  \widehat h_m^{-1}\norm{v-\widehat\Pi_{m-\c{n}}v}{L^2(\widehat\omega_m^n(z))}^2.
\end{align*}
For $z\in\NN_\ell$ with $\level_\ell(z) = m$,
Lemma~\ref{lem:unifEquivalence} states $z\in\widehat\NN_m$. 
This and $\#\widetilde\KK_m(z)\le\c{Cbound}$ from Lemma~\ref{lem:Kbound} give
\begin{align*}
  \sum_{m=0}^\infty \sum_{z\in\NN_L} \sum_{\ell\in \widetilde\KK_m(z)} 
  \widehat h_m^{-1} \norm{v-\widehat\Pi_{m-\c{n}}v}{L^2(\widehat\omega_m^n(z))}^2
  &= \sum_{m=0}^\infty \sum_{z\in\NN_L\cap \widehat\NN_m}
  \sum_{\ell\in\widetilde\KK_m(z)} 
  \widehat h_m^{-1} \norm{v-\widehat\Pi_{m-\c{n}}v}{L^2(\widehat\omega_m^n(z))}^2
  \\ 
  &\lesssim \sum_{m=0}^\infty \sum_{z\in\NN_L\cap \widehat\NN_m}
  \widehat h_m^{-1} \norm{v-\widehat\Pi_{m-\c{n}}v}{L^2(\widehat\omega_m^n(z))}^2
  \\ 
  &\leq \sum_{m=0}^\infty \sum_{z\in\widehat\NN_m}
  \widehat h_m^{-1}\norm{v-\widehat\Pi_{m-\c{n}}v}{L^2(\widehat\omega_m^n(z))}^2.
\end{align*}
Uniform $\gamma$-shape regularity of $\widehat\TT_m$ and the definition $\widehat\Pi_m=\widehat\Pi_0$ 
for $m<0$ yield
\begin{align*}
 \sum_{m=0}^\infty \sum_{z\in\widehat\NN_m}
  \widehat h_m^{-1}\norm{v-\widehat\Pi_{m-\c{n}}v}{L^2(\widehat\omega_m^n(z))}^2
 &\lesssim \sum_{m=0}^\infty\widehat h_m^{-1}\norm{v-\widehat\Pi_{m-\c{n}}v}{L^2(\Gamma)}^2
 \lesssim \sum_{m=0}^\infty\widehat h_m^{-1}\norm{v-\widehat\Pi_mv}{L^2(\Gamma)}^2
\end{align*}
We combine the last four estimates with 
Lemma~\ref{lem:h12normest} and norm equivalence on $H^{1/2}(\Gamma)$ to see
\begin{align*}
 \sum_{\ell=0}^L \sum_{z\in\widetilde\NN_\ell} \enorm{\widetilde v_z^\ell}^2
 \lesssim \sum_{m=0}^\infty\widehat h_m^{-1}\norm{v-\widehat\Pi_mv}{L^2(\Gamma)}^2
 \lesssim \norm{v}{H^{1/2}(\Gamma)}^2
 \simeq\enorm{v}^2.
\end{align*}
This and Lemma~\ref{lem:stabledec} then conclude the proof of the lower bound 
in~\eqref{eq:prop:PAS}.
\end{proof}

%%%%%%%%%%%%%%%%%%%5
% Upper bound
% 
%
% !TEX root = hypsingAS.tex

%--------------------------------------------------------------------------------
\subsection{Proof of Proposition~\ref{prop:as}, upper bound}
\label{section:upper bound}
%--------------------------------------------------------------------------------
In this section, we prove the upper bound in the spectral equivalence
estimate~\eqref{eq:prop:PAS}.

Let $M: = \max_{z\in\NN_L} \level_L(z)$ denote the maximal level of all nodes
$z\in\NN_L$ and note that Lemma~\ref{lem:unifEquivalence} yields $\NN_L\subseteq\widehat\NN_M$
and hence $\XX^L\subseteq\widehat\XX^M$.
We rewrite the additive Schwarz operator $\PAStilde^L$ as
\begin{align}\label{eq:pasAlt}
  \PAStilde^L = \sum_{\ell=0}^L \sum_{z\in\widetilde\NN_\ell}
  \prec_z^\ell  =  \sum_{m=0}^M \widetilde\QQ_m^L
  \quad\text{with}\quad
  \widetilde\QQ_m^L :=
  \sum_{\ell=0}^L \sum_{\substack{z\in\widetilde\NN_\ell \\
  \level_\ell(z) = m}} \prec_z^\ell.
\end{align}
There holds the following strengthened Cauchy-Schwarz inequality.

\begin{lemma}\label{lem:hypscs}
For all $0\leq m\leq M$, $k\leq m+1$
\begin{align}\label{eq:hypscs}
  0 \le \edual{\widetilde\QQ^L_m \widehat v^k}{\widehat v^k} 
 \le\c{csu}
  2^{-(m+1-k)} \enorm{\widehat v^k}^2
  \quad\text{for all }\widehat v^k\in\widehat\XX^k.
\end{align}
The constant $\setc{csu}>0$ depends only on $\Gamma$ and the initial triangulation $\TT_0$.
\end{lemma}

\begin{proof}
By definition of $\widetilde\QQ_m^L$, it holds
\begin{align}\label{eq:ub1}
  \edual{\widetilde\QQ_m^L\widehat v^k}{\widehat v^k}
 &= \sum_{\ell=0}^L \sum_{\substack{z\in\widetilde\NN_\ell \\ \level_\ell(z)=m}}
  \edual{\prec_z^\ell \widehat v^k}{\widehat v^k} 
 = \sum_{\ell=0}^L \sum_{\substack{z\in\widetilde\NN_\ell \\ \level_\ell(z)=m}}
  \enorm{\prec_z^\ell \widehat v^k}^2 \ge0.
\end{align}
Let $z\in\widetilde\NN_\ell$ with $\level_\ell(z)=m$.
Lemma~\ref{lem:unifEquivalence} states $h_\ell(z) \simeq \widehat h_m$ and 
$z\in\widetilde\NN_\ell\cap\widehat\NN_m$.
From the representation~\eqref{eq:Pz} of $\prec_z^\ell$, we get
\begin{align}\label{eq:scs3}
  \edual{\prec_z^\ell \widehat v^k}{\widehat v^k} = 
  \frac{\edual{\widehat v^k}{\eta_z^\ell}^2}{\enorm{\eta_z^\ell}^2}
  \lesssim  \frac{\dual{\hyp\widehat v^k}{\eta_z^\ell}_\Gamma^2 +
  \dual{\widehat v^k}1_\Gamma^2\dual{\eta_z^\ell}1_\Gamma^2}{\enorm{\eta_z^\ell}^2}.
\end{align}
According to, e.g.,~\cite[Theorem~4.8]{amt99}, it holds $h_\ell(z)^s\,\norm{\eta_z^\ell}{H^s(\Gamma)}\simeq\norm{\eta_z^\ell}{L^2(\Gamma)}$, where the hidden constants depends on $\Gamma$, 
$0\leq s\leq 1$, and $\gamma$-shape regularity of $\TT_\ell$.
%Lemma~\ref{lemma:scaling} yields $h_\ell(z)\,\enorm{\eta_z^\ell}^2 
%\simeq h_\ell(z)\,\norm{\eta_z^\ell}{H^{1/2}(\Gamma)}^2 
%\simeq \norm{\eta_z^\ell}{L^2(\Gamma)}^2$.
With $h_\ell(z)\,\enorm{\eta_z^\ell}^2 
\simeq h_\ell(z)\,\norm{\eta_z^\ell}{H^{1/2}(\Gamma)}^2 
\simeq \norm{\eta_z^\ell}{L^2(\Gamma)}^2$,
the Cauchy-Schwarz inequality thus gives
\begin{align*}
  \frac{\dual{\hyp\widehat v^k}{\eta_z^\ell}_\Gamma^2}{\enorm{\eta_z^\ell}^2} 
  \lesssim h_\ell(z) \norm{\hyp \widehat v^k}{L^2(\omega_\ell(z))}^2
  \lesssim \frac{\widehat h_{m}}{\widehat h_k}\, \norm{\widehat h_k^{1/2} 
\hyp\widehat v^k}{L^2(\omega_\ell(z))}^2
\lesssim 2^{-(m+1-k)}\,\norm{\widehat h_k^{1/2} \hyp\widehat v^k}{L^2(\omega_\ell(z))}^2.
\end{align*}
For the stabilization term, the same arguments together with
$\dual{\eta_z^\ell}1_\Gamma\le\norm{\eta_z^\ell}{L^2(\Gamma)}|\omega_\ell(z)|^{1/2}$ yield
\begin{align*}
 \frac{\dual{\widehat v^k}1_\Gamma^2\dual{\eta_z^\ell}1_\Gamma^2}{\enorm{\eta_z^\ell}^2}
 &\lesssim 
%h_\ell(z)\,|\omega_\ell(z)| \norm{\widehat v^k}{L^2(\Gamma)}^2
% \le 
 h_\ell(z) |\omega_\ell(z)| \norm{\widehat v^k}{H^{1/2}(\Gamma)}^2
\lesssim 2^{-(m+1-k)}\,|\omega_\ell(z)| \norm{\widehat v^k}{H^{1/2}(\Gamma)}^2,
\end{align*}
where the last estimate follows from 
$h_\ell(z)\simeq\widehat h_m \lesssim \widehat h_m/\widehat h_k = 2\cdot 2^{-(m+1-k)}$.
Combining these three estimates, we obtain
\begin{align*}
  \edual{\prec_z^\ell \widehat v^k}{\widehat v^k}
 \lesssim 2^{-(m+1-k)}\,\big(\norm{\widehat h_k^{1/2} \hyp\widehat v^k}{L^2(\omega_\ell(z))}^2
 + |\omega_\ell(z)| \norm{\widehat v^k}{H^{1/2}(\Gamma)}^2\big)
\end{align*}
By Lemma~\ref{lem:patchinclusion}, we have $\omega_\ell(z) \subseteq \widehat\omega_m^n(z)$. The representation of 
$\widetilde\QQ_m^L$ from~\eqref{eq:ub1} thus gives
\begin{align*}
  \edual{\widetilde\QQ_m^L\widehat v^k}{\widehat v^k}
& \lesssim2^{-(m+1-k)}
  \sum_{\ell=0}^L \sum_{\substack{z\in\widetilde\NN_\ell \\ \level_\ell(z)=m}}
   (\norm{\widehat h_k^{1/2}\hyp \widehat v^k}{L^2(\widehat\omega_m^n(z))}^2 +
  |\widehat\omega_m^n(z)| \norm{\widehat v^k}{H^{1/2}(\Gamma)}^2).
\end{align*}
By definition~\eqref{def:KK} of $\widetilde\KK_m(z)$ and Lemma~\ref{lem:Kbound}, the double sum can be
rewritten and further estimated by
\begin{align}\label{eq:scs5}
\begin{split}
 \sum_{\ell=0}^L \sum_{\substack{z\in\widetilde\NN_\ell \\ \level_\ell(z)=m}}&
   \big(\norm{\widehat h_k^{1/2}\hyp \widehat v^k}{L^2(\widehat\omega_m^n(z))}^2 +
  |\widehat\omega_m^n(z)| \norm{\widehat v^k}{H^{1/2}(\Gamma)}^2\big)\\
  &= \sum_{z\in \widehat\NN_m\cap\NN_L} \sum_{\ell\in\widetilde\KK_m(z)}
  \big(\norm{\widehat h_k^{1/2}\hyp \widehat v^k}{L^2(\widehat\omega_m^n(z))}^2
  + |\widehat\omega_m^n(z)| \norm{\widehat v^k}{H^{1/2}(\Gamma)}^2\big) \\
  &\lesssim \sum_{z\in \widehat\NN_m\cap\NN_L}
  \big(\norm{\widehat h_k^{1/2}\hyp \widehat v^k}{L^2(\widehat\omega_m^n(z))}^2
  + |\widehat\omega_m^n(z)| \norm{\widehat v^k}{H^{1/2}(\Gamma)}^2\big) \\
  &\leq \sum_{z\in \widehat\NN_m}
  \big(\norm{\widehat h_k^{1/2}\hyp \widehat v^k}{L^2(\widehat\omega_m^n(z))}^2
  + |\widehat\omega_m^n(z)| \norm{\widehat v^k}{H^{1/2}(\Gamma)}^2\big)
\end{split}
\end{align}
By $\gamma$-shape regularity of $\widehat\TT_m$, it holds
\begin{align*}
 \sum_{z\in \widehat\NN_m}
  \big(\norm{\widehat h_k^{1/2}\hyp \widehat v^k}{L^2(\widehat\omega_m^n(z))}^2
  + |\widehat\omega_m^n(z)| \norm{\widehat v^k}{H^{1/2}(\Gamma)}^2\big)
 &\lesssim 
  \norm{\widehat h_k^{1/2}\hyp \widehat v^k}{L^2(\Gamma)}^2
  + \norm{\widehat v^k}{H^{1/2}(\Gamma)}^2.
\end{align*}
Recall stability $\hyp:H^1(\Gamma)\to L^2(\Gamma)$. Together with
an inverse estimate between $H^1(\Gamma)$ and $H^{1/2}(\Gamma)$ for
piecewise polynomials, see e.g.~\cite[Proposition~5]{hypsing3d}, we obtain
\begin{align}\label{eq:scs8}
 \norm{\widehat h_k^{1/2}\hyp \widehat v^k}{L^2(\Gamma)}^2
 = \widehat h_k\norm{\hyp \widehat v^k}{L^2(\Gamma)}^2
 \lesssim \widehat h_k\norm{\widehat v^k}{H^1(\Gamma)}^2
 \lesssim \norm{\widehat v^k}{H^{1/2}(\Gamma)}^2.
\end{align}
We note that the latter estimate does not only hold for uniform triangulations
$\widehat\TT_k$, but also for shape-regular triangulations and higher-order
polynomials~\cite[Corollary~2]{invest}.
Combining the last four estimates
with norm equivalence $\norm{\widehat v^k}{H^{1/2}(\Gamma)}\simeq\enorm{\widehat v^k}$, 
we conclude the proof.
\end{proof}

The rest of the proof follows along the lines
of the proof of~\cite[Lemma~2.8]{transtep96} and is given for completeness.
We note that Lemma~\ref{lem:hypscs} implies, in particular, that 
$\edual{\widetilde\QQ_m^L v}{w}$ defines a positive semi-definite and symmetric bilinear
form on $\widehat\XX^k$ for $k\le m+1$ and hence satisfies a Cauchy-Schwarz inequality.

\begin{proof}[Proof of upper bound in~\eqref{eq:prop:PAS}]
Let $\widehat\gal_m : H^{1/2}(\Gamma) \to \widehat\XX^m$ denote the Galerkin projection onto
$\widehat\XX^m$ with respect to the scalar product $\edual\cdot\cdot$, i.e.,
\begin{align}
  \edual{\widehat\gal_m v}{\widehat w^m} = \edual{v}{\widehat w^m} \quad\text{for all }
  \widehat w^m\in\widehat\XX^m.
\end{align}
Note that $\widehat \gal_m$ is the orthogonal projection onto $\widehat\XX^m$ with respect to
the energy norm $\enorm\cdot$. We set $\widehat\gal_{-1} := 0$.
For any $v\in\XX^L\subseteq \widehat\XX^M$, it holds
\begin{align}
  \widehat\gal_m v = \sum_{k=0}^m
  (\widehat\gal_k-\widehat\gal_{k-1}) v
  \quad\text{as well as}\quad
  \widehat\gal_Mv=v = \widehat\gal_{M+1}v.
\end{align}
Lemma~\ref{lem:hatunif} yields $\widetilde\QQ_m^Lv\in\widehat\XX^{m+1}$. The symmetry of the orthogonal projection $\widehat\gal_m$ hence shows
\begin{align*}
  \edual{\widetilde\QQ_m^Lv}{v}
 = \edual{\widetilde\QQ_m^Lv}{\widehat\gal_{m+1}v}
 &= \sum_{k=0}^{m+1}\edual{\prec_m^Lv}{(\widehat\gal_k-\widehat\gal_{k-1})v}
 \\&\le
 \sum_{k=0}^{m+1}\edual{\widetilde\QQ_m^Lv}{v}^{1/2}\edual{\widetilde\QQ_m^L(\widehat\gal_k-\widehat\gal_{k-1})v}{(\widehat\gal_k-\widehat\gal_{k-1})v}^{1/2},
\end{align*}
where we have used the Cauchy-Schwarz inequality for $\edual{\widetilde\QQ_m^L v}{w}$
with $w=(\widehat\gal_k-\widehat\gal_{k-1})v\in\widehat\XX^k$.
For the second scalar product, we apply Lemma~\ref{lem:hypscs} and obtain
\begin{align*}
 \edual{\widetilde\QQ_m^L(\widehat\gal_k-\widehat\gal_{k-1})v}{(\widehat\gal_k-\widehat\gal_{k-1})v}
 \lesssim 
2^{-(m+1-k)}\,\enorm{(\widehat\gal_k-\widehat\gal_{k-1})v}^2
 = 
 2^{-(m+1-k)}\edual{(\widehat\gal_k-\widehat\gal_{k-1})v}{v}
\end{align*}
With the  representation~\eqref{eq:pasAlt} of $\PAStilde^L$ and the
Young inequality, we infer
\begin{align*}
  \edual{\PAStilde^L v}{v} &= \sum_{m=0}^M \edual{\widetilde\QQ_m^L v}{v} \\
  &\leq  \frac\delta{2} \sum_{m=0}^M \sum_{k=0}^{m+1}
  2^{-(m+1-k)/2} \edual{\widetilde\QQ_m^Lv}v
  + \frac{\delta^{-1}}2 \sum_{m=0}^M\sum_{k=0}^{m+1}
  2^{-(m+1-k)/2} \edual{\widehat\gal_k-\widehat\gal_{k-1})v}v,
\end{align*}
for all $\delta>0$.
There holds $\sum_{k=0}^{m+1} 2^{-(m+1-k)/2} \leq \sum_{k=0}^\infty
2^{-k/2} =: K<\infty$. Changing the summation
indices in the second sum, we see
\begin{align*}
  \edual{\PAStilde^L v}v
 &\leq K\frac\delta{2} \, \sum_{m=0}^M \edual{\widetilde\QQ_m^Lv}v + 
 \frac{\delta^{-1}}2 \, \sum_{k=0}^{M+1} \sum_{m=k-1}^M 2^{-(m+1-k)/2} 
 \edual{(\widehat\gal_k-\widehat\gal_{k-1})v}v \\
  &\leq K\frac\delta{2} \, \sum_{m=0}^M  \edual{\widetilde\QQ_m^Lv}v + 
  K \frac{\delta^{-1}}2 \,\sum_{k=0}^{M+1}
  \edual{(\widehat\gal_k-\widehat\gal_{k-1})v}v \\
  &= K\frac\delta{2} \edual{\PAStilde^L v}v + K \frac{\delta^{-1}}2 \edual{v}v,
\end{align*}
where the final equality follows from the telescoping series and 
$\widehat\gal_{M+1}v=v$.
Choosing $\delta>0$ sufficiently small and absorbing the first-term on the right-hand 
side on the left, we conclude the upper bound in~\eqref{eq:prop:PAS}.
\end{proof}

% !TEX root = hypsingAS.tex

%--------------------------------------------------------------------------------
\section{Proof of Theorem~\ref{thm:gmld}}\label{section:proof:gmld}
%--------------------------------------------------------------------------------
\noindent
Clearly, the abstract analytical setting for additive Schwarz
operators from Section~\ref{sec:as} also applies for the operator
\begin{align}
  \PAS^L = \sum_{\ell=0}^L \sum_{z\in\NN_\ell} \prec_z^\ell
\end{align}
associated to the preconditioner defined in Section~\ref{section:main:gmld}.
We stress that the properties of $(\BB^L)^{-1}$ and $\PPAS$ follow in the same
way as for Theorem~\ref{thm:main}. It thus remains to provide a lower and upper
bound for the operator $\PAS^L$ in analogy to Proposition~\ref{prop:as} for
$\PAStilde^L$.
\begin{proposition}\label{prop:gmld}
 The operator $\PAS^L$ satisfies
\begin{align}\label{eq:prop:GMLD}
  c \, \enorm{v}^2 \leq \edual{\PAS^L v}v 
 \le C(L+1)\,\enorm{v}^2\quad\text{for all } v\in\XX^L.
\end{align}
The constants $c,C>0$ depend only on $\Gamma$ and the initial triangulation
$\TT_0$.
\end{proposition}
In principle, the proof follows the same lines as in
Section~\ref{section:lower bound}--\ref{section:upper bound}. We sketch the most important
modifications only. Details are left to the reader.

\subsection{Proof of lower bound in~\eqref{eq:prop:GMLD}}
By virtue of Lemma~\ref{lem:stabledec}, we need to construct a decomposition $v
= \sum_{\ell=0}^L \sum_{z\in\NN_\ell} v_z^\ell$ with $v_z^\ell \in \XX_z^\ell$
and $\sum_{\ell=0}^L \sum_{z\in\NN_\ell} \enorm{v_z^\ell}^2 \leq c^{-1}
\enorm{v}^2$, for all $v\in\XX^L$. Since $\widetilde\NN_\ell \subseteq
\NN_\ell$, we may rely on the same decomposition as in
Section~\ref{section:lower bound}. This concludes the proof with the same
constant $c>0$ for Proposition~\ref{prop:as} and Proposition~\ref{prop:gmld}.

\subsection{Proof of upper bound in~\eqref{eq:prop:GMLD}}
For the proof of the upper bound in~\eqref{eq:prop:GMLD}, we define the set
\begin{align}
  \KK_k(z) := \set{ \ell\in \{0,1,\dots,L\}}{z\in\NN_\ell
  \quad\text{and}\quad \level_\ell(z) = k}.
\end{align}
\begin{lemma}\label{lem:KboundGMLD}
  For all $z\in\NN_L$ and $k\in\N_0$ there holds
  \begin{align*}
    \# \KK_k(z) \leq L + 1.
  \end{align*}
\end{lemma}
\begin{proof}
Obviously, there are at most $L+1$ indices in the set $\KK_k(z)$.
\end{proof}
We proceed as in Section~\ref{section:upper bound} and provide a similar result
as in Lemma~\ref{lem:hypscs}, where, however, Lemma~\ref{lem:KboundGMLD} plays an
important role in the proof.
To that end, we define the operator $\QQ_m^L : \XX^L \to \widehat \XX^{m+1}$ by
\begin{align*}
  \QQ_m^L := \sum_{\ell=0}^L \sum_{\substack{z\in\NN_\ell\\ \level_\ell(z) = m}}
  \prec_z^\ell,
\end{align*}
and stress that
\begin{align*}
  \PAS^L = \sum_{m=0}^M \QQ_m^L.
\end{align*}
\begin{lemma}\label{lem:hypscsGMLD}
For all $0\leq m\leq M$, $k\leq m+1$
\begin{align}\label{eq:hypscsGMLD}
  0 \le \edual{\QQ^L_m \widehat v^k}{\widehat v^k} 
 \le\c{csugmld} \,(L+1)\,
  2^{-(m+1-k)} \enorm{\widehat v^k}^2
  \quad\text{for all }\widehat v^k\in\widehat\XX^k.
\end{align}
The constant $\setc{csugmld}>0$ depends only on $\Gamma$ and the initial triangulation $\TT_0$.
\end{lemma}
\begin{proof}
The proof follows the same lines as the proof of Lemma~\ref{lem:hypscs}. The
important modifications consist in replacing $\widetilde\NN_\ell$ by $\NN_\ell$,
$\widetilde\KK_m(z)$ by $\KK_m(z)$, and $\widetilde\QQ_m^L$ by $\QQ_m^L$. 
However, in estimate~\eqref{eq:scs5} a bound for the cardinality of the set $\widetilde\KK_m(z)$ enters. 
Clearly, we have to replace this bound by the bound for $\# \KK_m(z)$ from Lemma~\ref{lem:KboundGMLD}. Therefore, the
factor $L+1$ comes into estimate~\eqref{eq:hypscsGMLD}.
\end{proof}
The rest of the proof is a simple adaptation of Section~\ref{section:upper bound}, taking care of the additional factor $L+1$.
It is therefore left to the reader.

% !TEX root = hypsingAS.tex

\section{Extension to screen problems}\label{sec:screen}

\subsection{Continuous setting}
Let $\Gamma \subsetneqq \partial\Omega$ denote an open screen. By $\widetilde
H^{1/2}(\Gamma)$, we denote the space of $H^{1/2}(\partial\Omega)$ functions
which vanish outside of $\Gamma$.
It is known~\cite{steph87} that the hypersingular integral operator $\hyp : \widetilde
H^{1/2}(\Gamma) \to H^{-1/2}(\Gamma)$ from~\eqref{eq:hypsing} is a linear, bounded, symmetric, and
elliptic operator.
The definition 
\begin{align}
  \edual{v}w := \dual{\hyp u}w_\Gamma \quad\text{for all }v,w\in \widetilde
  H^{1/2}(\Gamma)
\end{align}
provides a scalar product on $\widetilde H^{1/2}(\Gamma)$, and the induced norm
$\enorm{v}^2 := \edual{v}v$ is an equivalent norm on $\widetilde
H^{1/2}(\Gamma)$.
Instead of~\eqref{eq:weakform}, we consider the variational formulation of the
hypersingular integral equation~\eqref{eq:hypsing}
\begin{align}
  \edual{u}v = \dual{f}v_\Gamma \quad\text{for all }v\in \widetilde
  H^{1/2}(\Gamma)
\end{align}
with right-hand side $f\in H^{-1/2}(\Gamma)$.
The lemma of Lax-Milgram proves that this formulation admits a unique solution
$u \in\widetilde H^{1/2}(\Gamma)$.

\subsection{Notations}
We use the same notations as in Section~\ref{section:main:triangulation}--\ref{section:main:hierarchy}.
The discrete space $\XX^\ell$ from~\eqref{eq:defdiscretespace} is replaced by
the definition
\begin{align}
  \XX^\ell := \SS_0^1(\TT_\ell) := \SS^1(\TT_\ell) \cap \widetilde H^{1/2}(\Gamma),
\end{align}
i.e.\ the space of piecewise linear and globally continuous functions, which
vanish outside the open boundary part $\Gamma$.
Moreover, the set $\NN_\ell$ now does not consist of all nodes of the
triangulation $\TT_\ell$, but only of the nodes which lie inside $\Gamma$, i.e.
\begin{align}\label{eq:defNscreen}
  \NN_\ell := \{z \text{ is a node of }\TT_\ell \text{ which lies inside }
  \Gamma \text{ but not on the relative boundary } \partial\Gamma.\}
\end{align} 
Clearly, $\# \NN_\ell = \dim(\SS_0^1(\Gamma))$.
Also note that the definition~\eqref{eq:defNtilde} of the set
$\widetilde\NN_\ell$ involves the set $\NN_\ell$.

The uniformly refined spaces $\widehat\XX^m$ are defined accordingly.

\subsection{Multilevel diagonal preconditioner}
We stick with the settings and notations as in
Section~\ref{section:main:precond}--\ref{section:main:gmld}.

\begin{theorem}
  Theorem~\ref{thm:main} and Theorem~\ref{thm:gmld} hold for screen problems.
\end{theorem}

\begin{proof}
First, we extend Theorem~\ref{thm:main} to problems on open
boundaries.
Note that the abstract analysis of additive Schwarz operators from
Section~\ref{section:proof} holds also for $\Gamma \subsetneqq \partial\Omega$.
In particular, we only need to prove the lower and upper bound from
Proposition~\ref{prop:as}.

We stress that Lemma~\ref{lem:unifEquivalence}--\ref{lem:hatunif} hold
accordingly if $H^{1/2}$ is replaced by $\widetilde H^{1/2}$:
\begin{itemize}
  \item Lemma~\ref{lem:unifEquivalence}, Lemma~\ref{lem:Kbound}--\ref{lem:hatunif}
  hold for this problem, since they are only related to the triangulations and
  the mesh-refinement procedure.
  \item Lemma~\ref{lem:h12normest} remains valid if $H^{1/2}$ is replaced by
    $\widetilde H^{1/2}$, since the equivalence~\eqref{eq:h12normest4} also holds for $\Gamma\subsetneqq \partial\Omega$
    and $H^{1/2}$ replaced by $\widetilde H^{1/2}$,
    see~\cite[Theorem~5]{amcl03}.
  \item Lemma~\ref{lem:szest} involves a variant of the Scott-Zhang operator,
    which has to be constructed appropriately, see Section~\ref{sec:sz} and the
    references therein.
    To this end, one may proceed as in~\cite{hypsing3d} with the restriction on the
    choice of the elements $T_z^\ell$ required here, so that Lemma~\ref{lem:szest}
    remains valid.
\end{itemize}
Altogether, the proof of the lower bound follows the same lines as in
Section~\ref{section:lower bound}. Clearly, the Sobolev space $H^{1/2}(\Gamma)$ has to
be replaced by $\widetilde H^{1/2}(\Gamma)$.
Note that our re-definition~\eqref{eq:defNscreen} of $\NN_\ell$ for $\Gamma
\subsetneqq \partial\Omega$ ensures that the hat-function $\eta_z^\ell$ vanishes outside of $\Gamma$ for all
$z\in\NN_\ell$. Thus, $\omega_\ell(z) = \supp(\eta_z^\ell) \subseteq \overline\Gamma$ and 
\begin{align}\label{eq:screenProof1}
  \norm{\eta_z^\ell}{\widetilde H^{1/2}(\Gamma)} =
  \norm{\eta_z^\ell}{H^{1/2}(\partial\Omega)}.
\end{align}
Therefore, estimate~\eqref{eq:lb2} holds, since
\begin{align}
  \enorm{\eta_z^\ell}^2 \simeq \norm{\eta_z^\ell}{\widetilde H^{1/2}(\Gamma)}^2
  = \norm{\eta_z^\ell}{H^{1/2}(\partial\Omega)}^2 \leq
   \norm{\eta_z^\ell}{L^2(\partial\Omega)}
   \norm{\eta_z^\ell}{H^1(\partial\Omega)} \lesssim h_\ell(z)^{d-2}.
\end{align}
The rest of the proof holds verbatim with the notational adaptations mentioned
above.

Finally, we stress that for the proof of the upper bound in
Proposition~\ref{prop:as}, one has to verify Lemma~\ref{lem:hypscs} only.
Due to~\eqref{eq:screenProof1} and~\cite[Theorem~4.8]{amt99}, we get
\begin{align*}
  h_\ell(z)^{1/2} \norm{\eta_z^\ell}{\widetilde H^{1/2}(\Gamma)} =
  h_\ell(z)^{1/2} \norm{\eta_z^\ell}{H^{1/2}(\partial\Omega)} \simeq
  \norm{\eta_z^\ell}{L^2(\partial\Omega)} = \norm{\eta_z^\ell}{L^2(\Gamma)}.
\end{align*}

The inverse-type estimate~\eqref{eq:scs8} for the
hypersingular integral operator $\hyp$ still hold true for open boundary parts
$\Gamma$ and $H^{1/2}(\Gamma)$ replaced by $\widetilde H^{1/2}(\Gamma)$.
Then, the same proof as for Lemma~\ref{lem:hypscs} can be used, if the stabilization terms from
equation~\eqref{eq:scs3} and the following equations are omitted.

The extension of Theorem~\ref{thm:gmld} to problems on open boundaries
can be obtained by the modifications from above and
Section~\ref{section:proof:gmld}. Details are left to the reader.
\end{proof}

%%%%%%%%%%%%%%%%%%%%%%%%%%%%%%%%%%%%%%%%%%%%%%%%%%%%%%%%%%%%%%%%%%%%%
% APPENDIX
%\appendix
%\input{10_appendix}
%%%%%%%%%%%%%%%%%%%%%%%%%%%%%%%%%%%%%%%%%%%%%%%%%%%%%%%%%%%%%%%%%%%%%
% Bibliography
\bibliographystyle{alpha}
\bibliography{literature}
%%%%%%%%%%%%%%%%%%%%%%%%%%%%%%%%%%%%%%%%%%%%%%%%%%%%%%%%%%%%%%%%%%%%%

\end{document}